\documentclass[11pt]{article}
\usepackage[tbtags]{amsmath}
\usepackage{cases}
\usepackage{mathrsfs}
\usepackage{amsfonts}
\usepackage{amsmath,amsthm,amssymb}
\usepackage{authblk}

\usepackage{color}

\theoremstyle{plain}
\newtheorem{theorem}{Theorem}[section]
\newtheorem{lemma}{Lemma}[section]
\newtheorem{remark}{Remark}[section]

\newtheorem{corollary}{Corollary}[section]
\newenvironment{Proof of Proposition}{\begin{proof}[Proof of Proposition 2.1:]}{\end{proof}}
\newenvironment{Proof of Theorem 2.1}{\begin{proof}[Proof of Theorem 2.1:]}{\end{proof}}
\newenvironment{Proof of Theorem 3.1}{\begin{proof}[Proof of Theorem 3.1:]}{\end{proof}}
\newif \ifLastSection \LastSectionfalse

\numberwithin{equation}{section}

\begin{document}

\title{{\bf {Convergence to diffusion waves for solutions of 1D Keller-Segel model}}}

\author[1]{Fengling Liu}
\author[2]{Nangao Zhang}
\author[3]{Changjiang Zhu\thanks{Corresponding author. \authorcr Email addresses: mathflliu@mail.scut.edu.cn (liu), mazhangngmath@mail.scut.edu.cn (zhang),\\    machjzhu@scut.edu.cn (zhu).}}
\affil[1,2,3]{\normalsize  School of Mathematics, South China University of Technology, Guangzhou 510641, P.R. China}

\date{}

\maketitle

\textbf{{\bf Abstract:}}
In this paper, we are concerned with the asymptotic behavior of solutions to the Cauchy problem (or initial-boundary value problem) of one-dimensional Keller-Segel model. For the Cauchy problem, we prove that the solutions time-asymptotically converge to the nonlinear diffusion wave whose profile is self-similar solution to the corresponding parabolic equation, which is derived by Darcy's law, as in \cite{Hsiao-Liu1992, Nishihara1996}. For the initial-boundary value problem, we consider two cases: Dirichlet boundary condition and null-Neumann boundary condition on $(u, \rho)$. In the case of Dirichlet boundary condition, similar to the Cauchy problem, the asymptotic profile is still the self-similar solution of the corresponding parabolic equation, which is derived by Darcy's law, thus we only need to deal with boundary effect. In the case of null-Neumann boundary condition, the global existence and asymptotic behavior of solutions near constant steady states are established. The proof is based on the elementary energy method and some delicate analysis of the corresponding asymptotic profiles.
\bigbreak \textbf{{\bf Key Words}:} Keller-Segel model, Darcy's law, nonlinear diffusion waves, asymptotic behavior.

\bigbreak  {\textbf{AMS Subject Classifications:} 85A25, 35L65, 35B40.}

\section{Introduction}
In 1970, E.F. Keller and L.A. Segel proposed a model to describe the aggregation process of cellular slime mold by the chemical attraction in their celebrated work \cite{Keller-Segel1970}. The model is now known as the Keller-Segel model, which can be written into the following
form:
\begin{equation}\label{1.1}
\left\{\begin{array}{l}
\partial_{t} u=a\Delta u-\kappa\nabla \cdot(u \nabla \rho),\\[2mm]
\partial_{t} \rho=b\Delta \rho+\mu u-\lambda \rho.
\end{array}\right.
\end{equation}
Here $u=u(x,t)$ denotes the density of bacteria and $\rho=\rho(x,t)$ denotes the concentraction of chemical substance which mediates the aggregation. $a$, $b$, $\lambda$, $\mu$ and $\kappa$ are positive constants. $a$ and $b$ are respectively the diffusion coefficients of bacteria and of chemical substance. $\lambda$ is a constant rate of decrease of the chemical substance. $\mu$ is a constant rate of the chemical substance production by the bacteria and $\kappa$ denotes the intensity of chemotaxis.

As an important biological model, the Keller-Segel model \eqref{1.1} has attracted great interest among many scholars and there have been many important developments.
In one-dimension case, for the Neumann initial-boundary value problem of model \eqref{1.1}, Osaki and Yagi in \cite{Osaki-Yagi2001} proved that the solutions of the model are global and uniformly bounded if the inital data are smooth sufficiently. Afterwards Iwasaki, Osaki and Yagi in \cite{Iwasaki-Osaki-Yagi2020} showed that every solution must converge to a stationary solution by using the Lojasiewicz-Simon gragient inequality of the Lyapunov function. For the Cauchy problem of model \eqref{1.1}, Nagai and Yamada in \cite{Nagai-Yamada2007} gave the large time behavior of bounded solutions.
In two-dimensioanl model, Nagai, Syukuinn and Umesako in \cite{Nagai-Syukuinn-Umesako2003} proved that every bounded solution to the Cauchy problem decays to zero as $t$ goes infinity with small data $u_{0}\in L^{1}(\mathbb{R}^{2}) \cap L^{\infty}(\mathbb{R}^{2})$ and $\nabla\rho_{0}\in L^{1}(\mathbb{R}^{2}) \cap L^{\infty}(\mathbb{R}^{2})$, and that the solution behaves like the heat kernel as the self-similar profile.
Calvez and Corrias in \cite{Calvez-Corrias2008} showed that under additional assumptions $u_0\log(1+|x|^2)\in L^1(\mathbb{R}^2)$ and $u_0\log u_0\in L^1(\mathbb{R}^2)$, any solution with $m(u_0; \mathbb{R}^2)<8\pi$ to the Cauchy problem exists globally in time.
Mizoguchi in \cite{Mizoguchi2013} proved that the global existence of solutions without the additional assumptions when $m(u_0; \mathbb{R}^2)<8\pi$. Later, the result in \cite{Biler-Guerra-Karch2015} meant that any critical mass may lead to a global-in-time solution.
In $n$-dimensional case ($n\geq 3$), Corrias and Perthame in \cite{Corrias-Perthame2006} proved existence of weak positive solutions to the Cauchy problem by taking small data $u_{0}\in L^{r}(\mathbb{R}^{n})$ and $\nabla\rho_{0}\in L^{n}(\mathbb{R}^{n})$ with $\frac{n}{2}<r\leq n$. If the norms $\|u_0\|_{L^{\frac{n}{2}}(\Omega)}$ and $\|\nabla\rho_0\|_{L^{n}(\Omega)}$ are suitably small, it was proved that the global bounded solutions exist, see \cite{Cao2015, Winkler2010} as well as references cited therein. In the case of large initial data, Winkler in \cite{Winkler2013} constructed a Lyapunov functional to prove that the solutions blow up in finite time. For other results related to \eqref{1.1}, we refer to the interesting works and the references therein, cf. \cite{Gajewski-Zacharias1998, Herrero-Velazquez1997, Horstmann2003, Nagai-Senba-Yoshida1997, Schaaf1985}.

However, we can observe that in the unbounded region, the above results require that the initial data be the same constant at infinity, specifically zero at infinity.  The different states of the initial data at infinity, as we know, have not been studied so far. In this paper, we will consider this problem by the methods introduced by Hsiao and Liu in \cite{Hsiao-Liu1992}. And it turns out that the solution to \eqref{1.1} will converge to the nonlinear diffusion waves.

Precisely, we shall restrict ourselves to the one-dimensional Keller-Segel model (cf. \cite{Iwasaki-Osaki-Yagi2020, Osaki-Yagi2001}):
\begin{equation}\label{1.2}
\left\{\begin{array}{l}
u_t
   -au_{xx}
   +\kappa[u \rho_x]_x=0,\\[2mm]
\rho_t
    -b\rho_{xx}
    +\lambda\rho
    -\mu u=0,\quad x\in \Omega,\quad t>0,
 \end{array}
        \right.
\end{equation}
where $\Omega=\mathbb{R}$  or $\mathbb{R}^+$. And we will consider the Cauchy problem on $\mathbb{R}$ and the initial-boundary value problem on $\mathbb{R}^+$ of \eqref{1.2}.

For the Cauchy problem, the initial data are given by
\begin{equation}\label{1.3}
\left(u,\rho)\right(x,0)=\left(u_{0},\rho_{0}\right)(x)\rightarrow \left(u_{\pm}, \rho_{\pm}\right), \ \ \ \ \text{as}\quad x\rightarrow \pm\infty,
\end{equation}
where $u_{+} \neq u_{-}$, $\rho_{+} \neq \rho_{-}$ and $u_{\pm}=\frac{\lambda}{\mu}\rho_{\pm}$.

For the initial-boundary value problem, we consider the initial data
\begin{equation}\label{1.6}
\left(u,\rho)\right|_{t=0}=\left(u_{0},\rho_{0}\right)(x) \rightarrow \left(u_{+}, \rho_{+}\right), \ \ \ u_{+}>0, \ \ u_{+}=\frac{\lambda}{\mu}\rho_{+}, \ \ \text{as}\quad x\rightarrow +\infty,
\end{equation}
and one of the following boundary conditions:

(1) Dirichlet boundary condition
\begin{equation}\label{1.7}
u|_{x=0}=\beta, \quad \rho|_{x=0}=\frac{\mu \beta}{\lambda},
\end{equation}
where the constant $\beta$ takes a value on $[u_-,u_+]$ if $u_-<u_+$ (or $[u_+,u_-]$ if $u_->u_+$);

(2) Null-Neumann boundary condition
\begin{equation}\label{1.8}
u_x|_{x=0}=0, \quad \rho_x|_{x=0}=0.
\end{equation}

We expect to prove that the solutions of the Cauchy problem \eqref{1.2}, \eqref{1.3} and the initial-boundary value problem \eqref{1.2}, \eqref{1.6}, \eqref{1.7} (or \eqref{1.2}, \eqref{1.6}, \eqref{1.8}) converge to the nonlinear diffusion waves as in Hsiao and Liu \cite{Hsiao-Liu1992}.

Regarding the convergence theory on the nonlinear diffusion waves, it is necessary for us to briefly review its development history. For the Cauchy problem, it was studied by Hsiao and Liu in \cite{Hsiao-Liu1992} for the first time. They proved that the solutions of the Cauchy problem for $p$-system with damping converge time-asymptotically to the nonlinear diffusion waves whose profile is self-similar solution to the corresponding parabolic equation, which is derived by Darcy's law.
Then, by more detailed and accurate energy estimates, Nishihara in \cite{Nishihara1996} generalized the result of Hsiao and Liu in \cite{Hsiao-Liu1992}, and obtained a more precise convergence rates.
Furthermore, by constructing an appropriate approximate Green function with the energy method together, Nishihara, Wang and Yang in \cite{Nishihara-Wang-Yang2000} further improved the convergence rates, which is optimal in the sense comparing with the heat equation.
 For other results, we refer to \cite{Geng-Wang2010, Huang-Pan2003, Huang-Pan2006, Mei2010, Wang-Yang2003}, and the references therein.
 For the initial-boundary value problem on a half line $\mathbb{R}^+$, Nishihara and Yang in \cite{Nishihara-Yang1999} considered the asymptotic behavior of solutions of the initial-boundary value problem on $\mathbb{R}^+$ to the equations of $p$-system with linear damping, and obtained the $L^2$ and $L^\infty$ convergence rates. Later, Marcati, Mei and Rubino in \cite{Marcati-Mei-Rubino2005} improved the $L^2$ convergence rates of \cite{Nishihara-Yang1999}.  For other results, see \cite{Jiang-Zhu2009, Lin-Lin-Mei2010, Ma-Mei2010, Marcati-Mei2000} and references cited therein.

Motivated by these preceding results, we shall prove, for the Cauchy problem \eqref{1.2}-\eqref{1.3}, that the solutions globally exist and time-asymptotically converge to the nonlinear diffusion waves, which is self-similar solution to the corresponding parabolic equation given by Darcy's law.
For the initial-boundary value problem, we consider two cases: Dirichlet boundary condition \eqref{1.7} or null-Neumann boundary condition \eqref{1.8}. In the case of Dirichlet boundary condition \eqref{1.7}, similar to the Cauchy problem, the asymptotic profile is still the self-similar solution of the corresponding parabolic equation, thus we only need to deal with boundary effect. In the case of null-Neumann boundary condition \eqref{1.8}, the global existence and asymptotic behavior of solutions near constant steady states are established.

Finally, we briefly give some remarks on our problem and review some key analytical techniques. Firstly, for the Cauchy problem,
the main difficult step lies in obtaining the zero-order energy estimates. On the one hand, we find that it is difficult to handle the bad term $\int_0^t\int_{\mathbb{R}}wz\mathrm{d}x\mathrm{d}\tau$ since we can not obtain the uniform-in-time estimates for $\int_0^t\int_{\mathbb{R}}w^2\mathrm{d}x\mathrm{d}\tau$ and $\int_0^t\int_{\mathbb{R}}z^2\mathrm{d}x\mathrm{d}\tau$. In order to overcome such a difficulty, we try to use the structure of the reformulated equations to produce a time-space integrable good term $\int_0^t\int_{\mathbb{R}}(\lambda w-\mu z)^2\mathrm{d}x\mathrm{d}\tau$ (see \eqref{2.3.5}-\eqref{2.3.7} in Lemma \ref{Lemma 2.3}).
On the other hand, we need to treat the term $\int_0^t\int_{\mathbb{R}}z^2\bar\rho_x^2\mathrm{d}x\mathrm{d}\tau$.
By a heuristic analysis, we realize that $\int_0^t\int_{\mathbb{R}}z^2\bar\rho_x^2\mathrm{d}x\mathrm{d}\tau$ can be transformed into $\int_0^t\int_{\mathbb{R}}z^2\widetilde\omega^2\mathrm{d}x\mathrm{d}\tau$ (see \eqref{2.3.15} in Lemma \ref{Lemma 2.3}), and then we can just control it by using  new estimate based on the argument in \cite{Huang-Li-Matsumura2010}. See Lemma \ref{Lemma 2.2} and Corollary \ref{Corollary 2.1} for details.
These two aspects are very vital for obtaining the zero-order energy estimates, and are conducive to obtaining the higher-order energy estimates. One can see the details of the zero-order energy estimates in Lemma \ref{Lemma 2.3}.

Secondly, in the case of Dirichlet boundary condition, we construct still the self-similar diffusion waves. Therefore, the zero-order and first-order energy estimates are similar to Cauchy problem, and we need only to consider the treatment of the boundary effect.
However, for the second-order energy estimates, due to the difficulties in dealing with some boundary terms, we no longer use the methods applied in the Cauchy problem but mainly use the structure of the reformulated equations \eqref{3.1.2} to close the {\it a priori} assumption. See Lemma \ref{Lemma 3.6}-Lemma \ref{Lemma 3.7} for more details.
In the case of Neumann boundary condition, we consider the global existence and asymptotic behavior of solutions near constant steady states $(u_+, \rho_+)$. The analysis is also quite similar to Cauchy problem.

The rest of the paper is organized as follows. In Section \ref{S2}, we consider the Cauchy problem for the system \eqref{1.2}. In Section \ref{S2.1}, the Cauchy problem is reformulated and main result will be stated. In Section \ref{S2.2}, we prepare some preliminaries, which will be useful in the proof of our theorem. Section \ref{S2.3} is devoted to the proof of our theorem. In Section \ref{S3}, we show that the corresponding initial-boundary value problem admits a unique global smooth solution. In Subsection \ref{S3.1}, we will obtain the convergence in the case of Dirichlet boundary condition; In Subection \ref{S3.2}, we will study the null-Neumann boundary problem.

{\bf Notations}: Hereafter, $C$ denotes some generic positive constants which are only dependent of the initial data and may vary from line to line. $C_\eta$ denotes the generally large positive constant depending on $\eta$. $L^{p}=L^{p}(\Omega) ~ (1 \leq p \leq \infty)$ denotes the Lebesgue space with the norm
\begin{equation}\notag
 \|f\|_{L^p}=\left(\int_{\Omega}|f(x)|^p\mathrm{d}x\right)^{\frac{1}{p}},\quad 1\leq p<\infty,
 \end{equation}
 \begin{equation}\notag
 \|f\|_{L^{\infty}}=\sup_{\Omega}|f(x)|.
 \end{equation}
For any integer $m \geq 0$, $H^{m}(\Omega)$ denotes the usual Sobolev space with the norm
\begin{equation}\notag
\|f\|_m=\left(\sum_{k=0}^m\|\partial_{x}^{k}f\|^2\right)^{\frac{1}{2}},
\end{equation}
where $\|\cdot\|=\|\cdot\|_{0}=\|\cdot\|_{L^{2}(\Omega)}$, $\Omega=\mathbb{R}$  or $\mathbb{R}^+$.

\section{Cauchy problem}\label{S2}
{\numberwithin{equation}{subsection}

\subsection{Reformulation of the Cauchy problem and main result}\label{S2.1}
We first reformulate the Cauchy problem \eqref{1.2}-\eqref{1.3}. From Darcy's law and asymptotic analysis, we notice that the first term $\rho_t$ and the second term $-b\rho_{xx}$ have a faster time-decay with respect to the term $\lambda\rho$. Therefore, we expect the solutions of \eqref{1.2} time-asymptotically behave as those of the following system
\begin{equation}\label{2.1.1}
\left\{\begin{array}{l}
\bar u_t
  =\displaystyle
   a\bar u_{xx}
   -\kappa[\bar u \bar \rho_x]_x,\\[2mm]
   \lambda\bar{\rho}-\mu\bar{u}=0,
   \end{array}
        \right.
\end{equation}
or
\begin{equation}\label{2.1.2}
\left\{\begin{array}{l}
\bar u_t
   -
 [(a-\frac{\kappa\mu}{\lambda}\bar u )\bar u_x]_x=0,\\[2mm]
\bar{\rho}
   =\displaystyle
    \frac{\mu}{\lambda}\bar{u}.
 \end{array}
        \right.
\end{equation}
Motivated by \cite{Hsiao-Liu1992,Nishihara1996}, we denote $\bar{u}$ by any solutions of \eqref{2.1.1} with the same end states as $u(x,0)$:
\begin{equation}\label{2.1.3}
\bar u(\pm\infty,t)=u_\pm,
\end{equation}\\
and set
\begin{equation}\label{2.1.4}
\bar \rho(\pm\infty,t)=\frac{\mu}{\lambda}u_{\pm}=\rho_\pm,
\end{equation}\\
due to the Darcy's law.

Combining \eqref{1.2} and \eqref{2.1.1}, we get
\begin{equation}\label{2.1.5}
\left\{\begin{array}{l}
(u-\bar{u})_t
   -a(u-\bar{u})_{xx}
   +\kappa[u \rho_x]_x-\kappa[\bar u \bar \rho_x]_x=0,\\[2mm]
\rho_t
    -b\rho_{xx}
    +\lambda(\rho-\bar{\rho})
    -\mu(u-\bar{u})=0.
 \end{array}
        \right.
\end{equation}
Setting the perturbation\\
\begin{equation}\label{2.1.6}
(w,z)(x,t)=(\rho-\bar{\rho},u-\bar u)(x,t),
\end{equation}
we have the reformulated problem
\begin{equation}\label{2.1.7}
\left\{\begin{array}{l}
z_t
   -az_{xx}
   +\kappa[(z+\bar{u})w_x+z\bar \rho_x]_x=0,\\[2mm]
w_t
    -bw_{xx}
    +\lambda w
    -\mu z+\bar\rho_t
    -b\bar\rho_{xx}=0,\\[2mm]
 \end{array}
        \right.
\end{equation}
with initial data
\begin{equation}\label{2.1.8}
\left(w,z)\right|_{t=0}=\left(w_{0},z_{0}\right)(x) \rightarrow 0\quad
 \text {as}
  \quad
  x\rightarrow \pm\infty.
\end{equation}
\begin{theorem}\label{Thm 2.1} {\bf (Cauchy problem).}
Suppose that both $\delta :=\left|\rho_{+}-\rho_{-}\right|+\left|u_{+}\right|+\left|u_{-}\right|$ and $\left\|w_{0}\right\|_{2}+\left\|z_{0}\right\|_{2}$ are sufficiently small. Then, the Cauchy problem \eqref{2.1.7}-\eqref{2.1.8} exists a unique time-global solutions $(w,z)(x,t)$, which satisfies
\begin{equation}\notag
w \in W^{i, \infty}\left([0, \infty) ; H^{2-i}\right), \quad i=0,1,2,~~~
z \in W^{i, \infty}\left([0, \infty) ; H^{2-i}\right), \quad i=0,1,2,	
\end{equation}

and
\begin{equation}\label{2.1.9}
\begin{split}
&\sum_{k=0}^{2}(1+t)^{k}\left(\|\partial_{x}^{k} w(t)\|^{2}+\|\partial_{x}^{k} z(t)\|^{2}\right)
  +(1+t)^{2}\left(\|w_t(t)\|^{2}+\|z_t(t)\|^{2}\right)
   \\
   &
    +
    \int_{0}^{t}\bigg[\sum_{j=0}^{2}(1+\tau)^{j}\left(\|\partial_{x}^{j+1} w(\tau)\|^{2}+\|\partial_{x}^{j+1} z(\tau)\|^{2}+\|\partial_x^j(\lambda w-\mu z)(\tau)\|^2\right)\\
    &+(1+\tau)^{2}\left(\|w_{xt}(\tau)\|^{2}+\|z_{xt}(\tau)\|^{2}+\|(\lambda w_t-\mu z_t)(\tau)\|^2\right)\bigg]d \tau\\
     \leq&
       C\left(\|w_{0}\|_{2}^{2}+\|z_{0}\|_{2}^{2}+\delta\right).
\end{split}
\end{equation}
\end{theorem}

\subsection{Preliminaries}\label{S2.2}

In this subsection, we are going to introduce some fundamental properties of the nonlinear diffusion waves $(\bar u,\bar \rho)$ and some elementary inequalities, which will play an important role later.

Firstly, combining $\eqref{2.1.2}_1$ with \eqref{2.1.3}, we have
\begin{equation}\label{2.2.1}
\left\{\begin{array}{l}
\bar u_t
  =\displaystyle
   (f(\bar{u})\bar u_{x})_x,\\[2mm]
\bar u(\pm\infty,t)
   =\displaystyle
    u_\pm,
 \end{array}
        \right.
\end{equation}
where $f(\bar{u})=a-\frac{\kappa \mu}{\lambda}\bar{u}$. From the previous works in \cite{Atkinson-Peletier1974,van Duyn-Peletier1977}, we can know that  \eqref{2.2.1} has a unique self-similar solution called nonlinear diffusion wave in the form
\begin{equation}\label{2.2.2}
\left\{\begin{array}{l}
\bar u(x,t)
  =\displaystyle
  \phi\left(\frac{x}{\sqrt{1+t}}\right):=\phi(\xi),
   \quad \xi \in \mathbb{R},
   \\[2mm]
\phi(\pm\infty)
   =\displaystyle
    u_\pm.
 \end{array}
        \right.
\end{equation}
Plug $\eqref{2.2.2}_1$ into $\eqref{2.2.1}_1$ and integrate to obtain, for any $\xi_0$,
\begin{equation}\label{2.2.3}
\begin{split}
\phi^{\prime}(\xi)
=
\frac{\phi^{\prime}\left(\xi_{0}\right) f\left(\phi\left(\xi_{0}\right)\right)}{f(\phi(\xi))} {\rm e}^{-\int_{\xi_{0}}^{\xi} \frac{\eta}{2 f(\phi(\eta))} \mathrm{d} \eta},
\end{split}
\end{equation}
\begin{equation}\label{2.2.4}
\begin{split}
\phi(\xi)
&=\phi\left(\xi_{0}\right)+\int_{\xi_{0}}^{\xi} \frac{\phi^{\prime}\left(\xi_{0}\right) f\left(\phi\left(\xi_{0}\right)\right)}{f(\phi(\eta))} {\rm e}^{-\int_{\xi_{0}}^{\xi}\frac{s}{2 f(\phi(s))}\mathrm{d}s}\mathrm{d}\eta\\
&=\phi\left(\xi_{0}\right)+\int_{\xi_{0}}^{\xi} \phi^{\prime}(\eta) \mathrm{d} \eta.		
\end{split}
\end{equation}
It has been shown that $\eqref{2.2.3}$ with boundary condition $\eqref{2.2.2}_2$ has a unique solution and that is strictly monotone increasing if $u_{+} > u_{-}$ and decreasing if $u_{+} < u_{-}$ in \cite{Atkinson-Peletier1974,van Duyn-Peletier1977}. According to $\eqref{2.2.4}$ and $f(\bar{u})>0$ ~(due to $|u_{\pm}|<\frac{a\lambda}{k\mu}$), we have
\begin{equation}\label{2.2.5}
|\phi^{\prime}(\xi)|\leq C{\rm e}^{-C_{0}{\xi}^2},
\end{equation}
for some $C_{0}$, $C>0$ depending on $u_\pm$. Moreover, $\phi^{\prime}(\xi_{0})$ has the following property that
\begin{equation}\notag
C_{1}\left|u_{+}-u_{-}\right| \leq\left|\phi^{\prime}(\xi_{0})\right| \leq C_{2}\left|u_{+}-u_{-}\right|,
\end{equation}
where $C_{1}$ and $C_{2}$ are positive constants depending on $u_\pm$.
Therefore, we can obtain
\begin{equation}\label{2.2.6}
|\phi^{\prime}(\xi)|\leq C|u_{+}-u_{-}|{\rm e}^{-C_{0}{\xi}^2}.	
\end{equation}
As one can see in \cite{Hsiao-Liu1992}, it is easy to prove that the self-similar solution $\phi(\xi)$ satisfies
\begin{equation}\label{2.2.7}
\sum_{k=1}^{4}\left|\frac{\mathrm{d}^{k}}{\mathrm{d} \xi^{k}} \phi(\xi)\right|+\left|\phi(\xi)-u_{+}\right|_{\{\xi>0\}}+\left|\phi(\xi)-u_{-}\right|_{\{\xi<0\}} \leq C\left|u_{+}-u_{-}\right| \mathrm{e}^{-C_{0} \xi^{2}},
\end{equation}
and $\bar{u}(x,t)$ satisfies the following dissipative properties:
\begin{equation}\label{2.2.8}
\begin{split}
&\bar{u}_{x}=\frac{\phi^{\prime}(\xi)}{\sqrt{1+t}}, \qquad
\bar{u}_{t}=-\frac{\xi \phi^{\prime}(\xi)}{2(1+t)}, \qquad
\bar{u}_{x x}=\frac{\phi^{\prime \prime}(\xi)}{1+t},\\
&\bar{u}_{x t}=-\frac{\phi^{\prime}(\xi)+\xi \phi^{\prime \prime}(\xi)}{2(1+t)^{\frac{3}{2}}},\quad
\bar{u}_{x x x}=\frac{\phi^{\prime \prime \prime}(\xi)}{(1+t)^{\frac{3}{2}}},\quad
\bar{u}_{t t}=\frac{\xi^{2} \phi^{\prime \prime}(\xi)+3 \xi \phi^{\prime}(\xi)}{4(1+t)^{2}},\\
&\bar{u}_{x x t}=-\frac{\xi \phi^{\prime \prime \prime}(\xi)+2 \phi^{\prime \prime}(\xi)}{2(1+t)^{2}},\qquad
\bar{u}_{x x x x}=\frac{\phi^{\prime \prime \prime \prime}(\xi)}{(1+t)^2}.\\
\end{split}
\end{equation}
From $\eqref{2.2.7}$ and $\eqref{2.2.8}$, we can prove that $\bar u(x,t)$ satisfies the following decay estimates.
\begin{lemma}\label{Lemma 2.1}
For each $p\in [1,\infty]$ is an integer, the self-similar solution of \eqref{2.2.1} holds that
\begin{equation}\notag
 \begin{split}
 &\min \left\{u_{+}, u_{-}\right\} \leq \bar{u}(x,t) \leq \max \left\{u_{+}, u_{-}\right\}, \\[2mm]
 & \left\|\partial_{x}^{k} \partial_{t}^{j} \bar{u}(t)\right\|_{L^{p}(\mathbb{R})} \leq C\left|u_{+}-u_{-}\right|(1+t)^{-\frac{k}{2}-j+\frac{1}{2p}}, \quad  k,j \geq 0,~~~ k+j\geq 1 .
\end{split}
\end{equation}
\end{lemma}
Due to $\bar{\rho}=\frac{\mu}{\lambda}\bar{u}$, the dissipative properties of $\bar{\rho}(x,t)$ are the same as $\bar{u}(x,t)$.

Next, we introduce an elementary inequality concerning the time-space integrable estimates with the square of the heat kernel as a weight function. For $\alpha>0$, we define
\begin{equation}\label{2.2.9}
\widetilde\omega(x, t)=(1+t)^{-\frac{1}{2}} \exp \left\{-\frac{\alpha x^{2}}{1+t}\right\},
\ \ \ \ \ g(x, t)=\int_{-\infty}^{x} \widetilde\omega(y, t) \mathrm{d} y.	
\end{equation}
It is easy to check that
\begin{equation}\label{2.2.10}
\widetilde\omega_{t}=\frac{1}{4\alpha}\widetilde\omega_{xx},\quad 4 \alpha g_{t}=\widetilde\omega_{x}, \quad\|g(\cdot, t)\|_{L^{\infty}}=\sqrt{\pi} \alpha^{-\frac{1}{2}}.	
\end{equation}
\begin{lemma}\label{Lemma 2.2} (see \cite{Huang-Li-Matsumura2010})
For $0<T \leq+\infty,$ assume that $h(x, t)$ satisfies
$$
h_{x} \in L^{2}\left(0,T; L^{2}(\mathbb{R})\right), \quad h_{t} \in L^{2}\left(0, T; H^{-1}(\mathbb{R})\right).
$$
Then the following estimate holds:
\begin{equation}\label{2.2.11}
\begin{split}
&\int_0^T \int_{\mathbb{R}} h^{2} \widetilde\omega^{2} \mathrm{d}x \mathrm{d}t \\
\leq&
 4 \pi\|h(0)\|^{2}+4 \pi \alpha^{-1} \int_{0}^{T}\left\|h_{x}(t)\right\|^{2} \mathrm{d}t+8 \alpha \int_{0}^{T}\left\langle h_{t}, h g^{2}\right\rangle_{H^{-1} \times H^{1}} \mathrm{d}t,
\end{split}
\end{equation}
where $\langle \cdot,\cdot \rangle$ denotes the inner product on $H^{-1}(\mathbb{R}) \times H^{1}(\mathbb{R})$.
\end{lemma}

\begin{corollary}\label{Corollary 2.1}
In addition to the condition of Lemma \ref{Lemma 2.2}, we assume further that $|u_\pm|<\delta\ll 1$, $\|h\|_{L^{\infty}(\mathbb{R})} \leq \varepsilon\ll 1$ and $h$ satisfies
\begin{equation}\label{2.2.16}
h_t=
   ah_{xx}
   -\kappa[(h+\bar{u})w_x+h\bar \rho_x]_x,\quad
   h(x,0)=h_{0}(x) \in L^2(\mathbb{R}),\quad h_x(+\infty,t)=0,	
\end{equation}
where $a$ and $\kappa$ are given positive constant, $\bar{u}$ and $\bar{\rho}$ are the self-similar solutions of \eqref{2.1.1}. Then there exists some positive constant $C>0$ such that
\begin{equation}\label{2.2.17}
\int_0^T\int_{\mathbb R} h^2 \widetilde\omega^2\mathrm{d}x \mathrm{d}t
\leq
C\int_0^T(\|h_x(\tau)\|^2+\|w_x(\tau)\|^2)\mathrm{d}\tau
+C\|h_0\|^2.	
\end{equation}
\end{corollary}
\begin{proof}
Using the integration by parts and \eqref{2.2.16}, we deduce
\begin{equation}\label{2.2.18}
\begin{split}
\left\langle h_{t}, h g^{2}\right\rangle_{H^{-1} \times H^{1}}&=\int_{\mathbb{R}}h_t hg^2\mathrm{d}x\\
&=\int_{\mathbb R}\bigg\{ah_{xx} -\kappa[(h+\bar{u})w_x+h\bar \rho_x]_x\bigg\}hg^2\mathrm{d}x\\
&=a\int_{\mathbb R}h_{xx}hg^2+\kappa\int_{\mathbb R}[(h+\bar{u})w_x+h\bar \rho_x](hg^2)_x\mathrm{d}x\\
&=-a\int_{\mathbb R}h_{x}^2g^2\mathrm{d}x-2a\int_{\mathbb R}h_xgh\widetilde\omega \mathrm{d}x+\kappa\int_{\mathbb R}[(h+\bar{u})w_x+h\bar \rho_x](h_xg^2+2gh\widetilde\omega)\mathrm{d}x\\
&\leq
-2a\int_{\mathbb R}h_xgh\widetilde\omega \mathrm{d}x+\kappa\int_{\mathbb R}(h+\bar{u})w_xh_xg^2\mathrm{d}x+2\kappa\int_{\mathbb R}(h+\bar{u})w_xgh\widetilde\omega\mathrm{d}x\\
&~~~+\kappa\int_{\mathbb R}h\bar \rho_xh_xg^2\mathrm{d}x+2\kappa\int_{\mathbb R}h\bar \rho_xgh\widetilde\omega\mathrm{d}x\\
&=\sum_{i=1}^{5}I_{i},
\end{split}
\end{equation}
where $I_{i} ~ (1\leq i\leq 5)$ corresponds to the terms on the right-hand side of the above inequality.

Now we estimate $I_{i} ~ (1\leq i\leq 5)$ term by term. By using \eqref{2.2.10}, Young inequality and the following inequality
\begin{equation}\label{2.2.19}
\begin{split}
\int_{\mathbb R} h^2 {\bar\rho_{x}}^2\mathrm{d}x
&\leq
C\int_{\mathbb R}h^2\left[(1+t)^{-\frac{1}{2}}|u_{+}-u_{-}|{\rm e}^{-\frac{C_{0} x^{2}}{1+t}}\right]^2\mathrm{d}x\\
&\leq
C{\delta}^2\int_{\mathbb R} h^2\widetilde\omega ^2 \mathrm{d}x,
\end{split}
\end{equation}
we can infer that
\begin{equation}\label{2.2.20}
 \begin{split}
 I_1&\leq \eta\int_{\mathbb R} h^2\widetilde\omega ^2 \mathrm{d}x+C_{\eta}\int_{\mathbb R} h_{x}^2 \mathrm{d}x,\\
 I_2&\leq C\int_{\mathbb R}(\|h\|_{L^{\infty}}+\|\bar u\|_{L^{\infty}})(w_{x}^2g^4+h_{x}^2)\mathrm{d}x\\
 &\leq C(\varepsilon+\delta)\int_{\mathbb R}w_{x}^2\mathrm{d}x+C(\varepsilon+\delta)\int_{\mathbb R}h_{x}^2\mathrm{d}x,\\
 I_3&\leq C\int_{\mathbb R}(\|h\|_{L^{\infty}}+\|\bar u\|_{L^{\infty}})(w_{x}^2g^2+h^2\widetilde\omega^2)\mathrm{d}x\\
 &\leq C(\varepsilon+\delta)\int_{\mathbb R}h^2\widetilde\omega^2\mathrm{d}x+C(\varepsilon+\delta)\int_{\mathbb R}w_{x}^2\mathrm{d}x,\\
 I_4&\leq C\int_{\mathbb R}h^2\bar \rho_{x}^2\mathrm{d}x+C\int_{\mathbb R}h_{x}^2g^4\mathrm{d}x\\
 &\leq C{\delta}^2\int_{\mathbb R} h^2\widetilde\omega ^2 \mathrm{d}x+C\int_{\mathbb R}h_{x}^2\mathrm{d}x,\\
 I_5&\leq \eta\int_{\mathbb R} h^2\widetilde\omega ^2 \mathrm{d}x+C_{\eta}\int_{\mathbb R} h^2\bar \rho_{x}^2 \mathrm{d}x\\
 &\leq \eta\int_{\mathbb R} h^2\widetilde\omega ^2 \mathrm{d}x+C_{\eta}{\delta}^2\int_{\mathbb R} h^2\widetilde\omega^2 \mathrm{d}x.
 \end{split}
\end{equation}
Here we have used $\eqref{2.1.1}_2$, $\eqref{2.2.2}_1$ and \eqref{2.2.7}-\eqref{2.2.9}.

Inserting the above inequalities into \eqref{2.2.18}, then integrating the resulting inequality with respect to $t$, using \eqref{2.2.11} and choosing $\delta$, $\eta$ sufficiently small, we can prove \eqref{2.2.17} easily. The proof of Corollary \ref{Corollary 2.1} is completed.
\end{proof}

\subsection{Proof of Theorem \ref{Thm 2.1}}\label{S2.3}
It is generally known that the global existence can be obtained by the classical continuation argument based on the local existence of solutions and {\it a priori} estimates. The local existence of solutions to the reformulated Cauchy problem $\eqref{2.1.7}$ and $\eqref{2.1.8}$ can be established by the standard iteration argument. The details are omitted. In order to prove Theorem \ref{Thm 2.1} for brevity, we only devote ourselves to obtaining the {\it a priori} estimates under the {\it a priori} assumption
\begin{equation}\label{2.3.1}
N(T):=\sup_{0\leq t\leq T}\left\{\sum_{k=0}^{2}(1+t)^{k}\left(\|\partial_{x}^{k} w(t)\|^{2}+\|\partial_{x}^{k} z(t)\|^{2}\right)\right\}\leq\varepsilon_0^2,
\end{equation}
for some $0<\varepsilon_0\ll 1$.

An easy application of Sobolev inequality for $L^{\infty}$, we can obtain inequalities
\begin{equation}\label{2.3.2}
\|\partial_{x}^{k} w(\cdot, t)\|_{L^{\infty}} \leq \sqrt2\varepsilon_0(1+t)^{-\frac{1}{4}-\frac{k}{2}}, \quad k=0,1,
\end{equation}
\begin{equation}\label{2.3.3}
\|\partial_{x}^{k} z(\cdot, t)\|_{L^{\infty}} \leq \sqrt2\varepsilon_0(1+t)^{-\frac{1}{4}-\frac{k}{2}}, \quad k=0,1,
\end{equation}
which will be used later.

Now we turn to establish \eqref{2.1.9}, which will be given by a series of lemmas.

\begin{lemma}\label{Lemma 2.3}
If $N(T) \leq \varepsilon_0^2$ and $\delta$ are small enough, it holds that
\begin{equation}\label{2.3.4}
\begin{split}
\|w(t)\|&^{2}+\|z(t)\|^{2}+\int_{0}^{t}\left(\left\|w_{x}(\tau)\right\|^{2}+\|z_{x}(\tau)\|^{2}+\|(\lambda w-\mu z)(\tau)\|^2\right) \mathrm{d} \tau\\
&\leq C\left(\left\|w_{0}\right\|^{2}+\left\|z_{0}\right\|^{2}+\delta\right),
\end{split}
\end{equation}
for $0 \leq t \leq T$.
\end{lemma}
\begin{proof}
Firstly, multiplying $\eqref{2.1.7}_2$ by $\lambda w$, integrating the resulting equality with respect to $x$ over $\mathbb{R}$, we obtain
\begin{equation}\label{2.3.5}
\frac{\rm d}{{\rm d}t}\int_{\mathbb{R}}\frac{\lambda w^2}{2} \mathrm{d}x+b\lambda\int_{\mathbb{R}}w_x^2\mathrm{d}x+\lambda^2\int_{\mathbb{R}}w^2\mathrm{d}x-\lambda\mu\int_{\mathbb{R}}wz\mathrm{d}x+\int_{\mathbb{R}}\lambda w(\bar\rho_t-b\bar\rho_{xx})\mathrm{d}x=0.
\end{equation}
Multiplying $\eqref{2.1.7}_2$ by $(-\mu z)$ and integrating it with respect to $x$ over $\mathbb{R}$, we have
\begin{equation}\label{2.3.6}
-\mu\int_{\mathbb{R}}w_tz\mathrm{d}x-b\mu\int_{\mathbb{R}}w_xz_x\mathrm{d}x-\lambda\mu\int_{\mathbb{R}}wz\mathrm{d}x+\mu^2\int_{\mathbb{R}}z^2\mathrm{d}x-\int_{\mathbb{R}}\mu z(\bar\rho_t-b\bar\rho_{xx})\mathrm{d}x=0.
\end{equation}
Next, we try to use the structure of Keller-Segel model to produce the good term: $\int_{\mathbb{R}}(\lambda w-\mu z)^2\mathrm{d}x$. By summing \eqref{2.3.5} and \eqref{2.3.6}, it follows that
\begin{equation}\label{2.3.7}
\begin{split}
&\frac{\rm d}{{\rm d}t} \int_{\mathbb{R}}
   \frac{\lambda{w}^2}{2}\mathrm{d}x
   +b\lambda\int_{\mathbb{R}} w_x^2\mathrm{d}x
   +\int_{\mathbb{R}} (\lambda w-\mu z)^2\mathrm{d}x\\
&+\int_{\mathbb{R}} (\lambda w-\mu z)(\bar\rho_t-b\bar\rho_{xx})\mathrm{d}x
=\mu\int_{\mathbb{R}}w_tz\mathrm{d}x+b\mu\int_{\mathbb{R}} {w_x}{z_x}\mathrm{d}x.
\end{split}
\end{equation}
By applying integration by parts and using $\eqref{2.1.7}_1$, one can obtain
\begin{equation}\label{2.3.8}
\begin{split}
\mu\int_{\mathbb{R}}w_tz\mathrm{d}x
&=\frac{\rm d}{{\rm d}t}\int_{\mathbb{R}}\mu wz\mathrm{d}x-\mu \int_{\mathbb{R}}wz_t\mathrm{d}x\\
&=\frac{\rm d}{{\rm d}t}\int_{\mathbb{R}}\mu wz\mathrm{d}x-\mu\int_{\mathbb{R}}w\bigg\{az_{xx} -\kappa[(z+\bar{u})w_x+z\bar \rho_x]_x\bigg\}\mathrm{d}x\\
&=\frac{\rm d}{{\rm d}t}\int_{\mathbb{R}}\mu wz\mathrm{d}x+a\mu\int_{\mathbb{R}}w_xz_x\mathrm{d}x-\kappa\mu\int_{\mathbb{R}}w_x[(z+\bar u)w_{x}+z\bar\rho_x]\mathrm{d}x\\
&=\frac{\rm d}{{\rm d}t}\int_{\mathbb{R}}\mu wz\mathrm{d}x+a\mu\int_{\mathbb{R}}w_xz_x\mathrm{d}x-\kappa\mu\int_{\mathbb{R}}(z+\bar u)w_x^2\mathrm{d}x-\kappa\mu\int_{\mathbb{R}}w_xz\bar\rho_x\mathrm{d}x.
\end{split}
\end{equation}
Taking \eqref{2.3.8} into \eqref{2.3.7}, we get
\begin{equation}\label{2.3.9}
\begin{split}
&\frac{\rm d}{{\rm d}t} \int_{\mathbb{R}}
  \left(\frac{\lambda w^2}{2}-\mu wz\right)\mathrm{d}x
   +{b\lambda}\int_{\mathbb{R}} w_x^2\mathrm{d}x
   +\int_{\mathbb{R}}(\lambda w-\mu z)^2\mathrm{d}x+\int_{\mathbb{R}} (\lambda w-\mu z)(\bar\rho_t-b\bar\rho_{xx})\mathrm{d}x\\
=&(a+b)\mu\int_{\mathbb{R}}{w_xz_x}\mathrm{d}x-\kappa\mu\int_{\mathbb{R}}(z+\bar u)w_x^2\mathrm{d}x-\kappa\mu\int_{\mathbb{R}}w_xz\bar\rho_x\mathrm{d}x.
\end{split}
\end{equation}
Now we need to estimate the last two terms on the right-hand side of \eqref{2.3.9}. By using Lemma \ref{Lemma 2.1} and \eqref{2.3.3}, we can derive
\begin{equation}\label{2.3.10}
\begin{split}
-\kappa\mu\int_{\mathbb{R}}(z+\bar u)w_x^2\mathrm{d}x
&\leq C\int_{\mathbb{R}}(\|z\|_{L^{\infty}}+\|\bar u\|_{L^{\infty}})w_{x}^2\mathrm{d}x\\
&\leq C(\varepsilon_0+\delta)\int_{\mathbb{R}}w_{x}^2\mathrm{d}x,
\end{split}
\end{equation}
and
\begin{equation}\label{2.3.11}
\begin{split}
-\kappa\mu\int_{\mathbb{R}}w_xz\bar\rho_x\mathrm{d}x
&\leq \eta \int_{\mathbb{R}}w_{x}^2\mathrm{d}x+C_\eta\int_{\mathbb{R}}z^2\bar{\rho}_{x}^2\mathrm{d}x\\
&\leq\eta \int_{\mathbb{R}}w_{x}^2\mathrm{d}x+C_\eta \delta^2\int_{\mathbb{R}}z^2\widetilde\omega^2\mathrm{d}x,
\end{split}
\end{equation}
where in the last inequality we have taken $h=z$ in \eqref{2.2.19}.
Hence, putting \eqref{2.3.10}-\eqref{2.3.11} into \eqref{2.3.9}, and using Young inequality, then choosing $\eta$ suitably small to arrive at
\begin{equation}\label{2.3.12}
\begin{split}
&\frac{\rm d}{{\rm d}t}\int_{\mathbb{R}} \left(\frac{\lambda w^2}{2}-\mu wz\right) \mathrm{d}x
+\frac{b\lambda}{4}\int_{\mathbb{R}}w_x^2\mathrm{d}x
+\frac{1}{2}\int_{\mathbb{R}}{(\lambda w-\mu z)^2}\mathrm{d}x\\
\leq& \frac{1}{2}\int_{\mathbb{R}} (\bar\rho_t
    -b\bar\rho_{xx})^2\mathrm{d}x
   +C\delta\int_{\mathbb{R}} z^2 \widetilde\omega^2\mathrm{d}x+C\int_{\mathbb{R}}z_{x}^2\mathrm{d}x.
\end{split}
\end{equation}
Now we only need to estimate the last term on the right-hand side of \eqref{2.3.12}.

Multiplying $\eqref{2.1.7}_1$ by $z$ and integrating it with respect to $x$, we obtain
\begin{equation}\label{2.3.13}
\frac{\rm d}{{\rm d}t} \int_{\mathbb{R}}\frac{ z^2}{2}\mathrm{d}x +a\int_{\mathbb{R}}z_{x}^2\mathrm{d}x=\kappa\int_{\mathbb{R}}(z+\bar u)w_xz_x\mathrm{d}x+\kappa\int_{\mathbb{R}}z\bar \rho_xz_x\mathrm{d}x.
\end{equation}
Similar to the treatment of \eqref{2.3.10} and \eqref{2.3.11}, one has
\begin{equation}\label{2.3.14}
\frac{\rm d}{{\rm d}t} \int_{\mathbb{R}}\frac{ z^2}{2}\mathrm{d}x +\frac{a}{2}\int_{\mathbb{R}}z_{x}^2\mathrm{d}x
\leq C(\varepsilon_0+\delta)\int_{\mathbb{R}}w_x^2\mathrm{d}x+C\delta\int_{\mathbb{R}} z^2 \widetilde\omega^2\mathrm{d}x.
\end{equation}
Multiplying \eqref{2.3.14} by a big positive constant $K$ and summing it to \eqref{2.3.12}, then using Lemma \ref{Lemma 2.1}, we have
\begin{equation}\label{2.3.15}
\begin{split}
&\frac{\rm d}{{\rm d}t} \int_{\mathbb{R}}
   \left(\frac{\lambda w^2}{2} +\frac{K z^2}{2}-\mu wz\right) \mathrm{d}x
   +\frac{b\lambda}{8}\int_{\mathbb{R}}w_x^2\mathrm{d}x+\frac{K a}{4}\int_{\mathbb{R}}z_x^2\mathrm{d}x
   +\frac{1}{2}\int_{\mathbb{R}}{(\lambda w-\mu z)^2}\mathrm{d}x\\
   \leq
&C\delta(1+t)^{-\frac{3}{2}}
    +C\delta\int_{\mathbb{R}} z^2 \widetilde\omega^2\mathrm{d}x.
\end{split}
\end{equation}
Integrating the above inequality with respect to $t$, and taking $h=z$ in \eqref{2.2.17}, we reach \eqref{2.3.4}. The proof of Lemma \ref{Lemma 2.3} is completed.
\end{proof}
\begin{lemma}\label{Lemma 2.4}
If $N(T) \leq \varepsilon_0^2$ and $\delta$ are small enough, it holds that
\begin{equation}\label{2.3.16}
\begin{split}
(1+t)&(\|w_{x}(t)\|^{2}+\|z_{x}(t)\|^{2})+\int_{0}^{t}(1+\tau)\left(\left\|w_{x x}(\tau)\right\|^{2}+\|z_{x x}(\tau)\|^{2}+\|(\lambda w_x-\mu z_x)(\tau)\|^2\right) \mathrm{d} \tau\\
&\leq C\left(\left\|w_{0}\right\|_{1}^{2}+\left\|z_{0}\right\|_{1}^{2}+\delta\right),
\end{split}
\end{equation}
for $0 \leq t \leq T$.
\end{lemma}
\begin{proof}
Differentiating \eqref{2.1.7} in $x$ to obtain
\begin{equation}\label{2.3.17}
\left\{\begin{array}{l}
z_{xt}
   -az_{xxx}
   +\kappa[(z+\bar{u})w_x+z\bar \rho_x]_{xx}=0,\\[2mm]
w_{xt}
    -bw_{xxx}
    +\lambda w_{x}
    -\mu z_{x}+\bar\rho_{xt}
    -b\bar\rho_{xxx}=0.\\[2mm]
 \end{array}
        \right.
\end{equation}
Firstly, multiplying $\eqref{2.3.17}_2$ by $\lambda w_x$ and integrating it with respect to $x$ over $\mathbb{R}$, we have
\begin{equation}\label{2.3.18}
\frac{\rm d}{{\rm d}t}\int_{\mathbb{R}}\frac{\lambda w_{x}^2}{2} \mathrm{d}x+b\lambda\int_{\mathbb{R}}w_{x x}^2\mathrm{d}x+\lambda^2\int_{\mathbb{R}}w_{x}^2\mathrm{d}x-\lambda\mu\int_{\mathbb{R}}w_xz_x\mathrm{d}x+\int_{\mathbb{R}}\lambda w_x(\bar\rho_{xt}-b\bar\rho_{xxx})\mathrm{d}x=0.
\end{equation}
Multiplying $\eqref{2.3.17}_2$ by $(-\mu z_x)$ and integrating it with respect to $x$ over $\mathbb{R}$, we have
\begin{equation}\label{2.3.19}
-\mu\int_{\mathbb{R}}w_{x t}z_x\mathrm{d}x-b\mu\int_{\mathbb{R}}w_{xx}z_{xx}\mathrm{d}x-\lambda\mu\int_{\mathbb{R}}w_xz_x\mathrm{d}x+\mu^2\int_{\mathbb{R}}z_{x}^2\mathrm{d}x-\int_{\mathbb{R}}\mu z_x(\bar\rho_{xt}-b\bar\rho_{xxx})\mathrm{d}x=0.
\end{equation}
Then summing \eqref{2.3.18} and \eqref{2.3.19} to obtain the good term $\int_{\mathbb{R}}(\lambda w_x-\mu z_x)^2\mathrm{d}x$, as follows:
\begin{equation}\label{2.3.20}
\begin{split}
&\frac{\rm d}{{\rm d}t} \int_{\mathbb{R}}
   \frac{\lambda w_{x}^2}{2}\mathrm{d}x
   +b\lambda\int_{\mathbb{R}}w_{xx}^2\mathrm{d}x
   +\int_{\mathbb{R}} (\lambda w_x-\mu z_x)^2\mathrm{d}x\\
&+\int_{\mathbb{R}} (\lambda w_x-\mu z_x)(\bar\rho_{xt}-b\bar\rho_{xxx})\mathrm{d}x
=\mu\int_{\mathbb{R}}w_{x t}z_x\mathrm{d}x+b\mu\int_{\mathbb{R}}w_{x x}z_{x x}\mathrm{d}x.
\end{split}
\end{equation}
By applying integration by parts and using $\eqref{2.3.17}_1$, one can obtain
\begin{equation}\label{2.3.21}
\begin{split}
\mu\int_{\mathbb{R}}w_{xt}z_x\mathrm{d}x
&=\frac{\rm d}{{\rm d}t}\int_{\mathbb{R}}\mu w_xz_x\mathrm{d}x-\mu\int_{\mathbb{R}}w_xz_{xt}\mathrm{d}x\\
&=\frac{\rm d}{{\rm d}t}\int_{\mathbb{R}}\mu w_xz_x\mathrm{d}x-\mu \int_{\mathbb{R}}w_x\bigg\{az_{xxx} -\kappa[(z+\bar{u})w_x+z\bar \rho_x]_{xx}\bigg\}\mathrm{d}x\\
&=\frac{\rm d}{{\rm d}t}\int_{\mathbb{R}}\mu w_xz_x\mathrm{d}x+a\mu\int_{\mathbb{R}}w_{xx}z_{xx}\mathrm{d}x-\kappa\mu\int_{\mathbb{R}}w_{xx}[(z+\bar u)w_x+z\bar\rho_x]_x\mathrm{d}x\\
&=\frac{\rm d}{{\rm d}t}\int_{\mathbb{R}}\mu w_xz_x\mathrm{d}x+a\mu\int_{\mathbb{R}}w_{xx}z_{xx}\mathrm{d}x
-\kappa\mu\int_{\mathbb{R}}(z+\bar u)w_{x x}^2\mathrm{d}x\\
 &~~~~-\kappa\mu\int_{\mathbb{R}}(z_x+\bar u_x)w_{x}w_{x x}\mathrm{d}x
 -\kappa\mu\int_{\mathbb{R}}z_x\bar\rho_xw_{x x}\mathrm{d}x
-\kappa\mu\int_{\mathbb{R}}z\bar \rho_{xx} w_{x x}\mathrm{d}x.
\end{split}
\end{equation}
Putting \eqref{2.3.21} into \eqref{2.3.20}, we get
\begin{equation}\label{2.3.22}
\begin{split}
&\frac{\rm d}{{\rm d}t} \int_{\mathbb{R}}
   \left(\frac{\lambda w_x^2}{2}-\mu w_xz_x\right)\mathrm{d}x
   +{b\lambda}\int_{\mathbb{R}}w_{xx}^2\mathrm{d}x
   +\int_{\mathbb{R}} (\lambda w_x-\mu z_x)^2\mathrm{d}x\\
   &~~+\int_{\mathbb{R}} (\lambda w_x-\mu z_x)(\bar\rho_{xt}-b\bar\rho_{xxx})\mathrm{d}x\\
=&(a+b)\mu\int_{\mathbb{R}}w_{xx}z_{x x}\mathrm{d}x-\kappa\mu\int_{\mathbb{R}}(z+\bar u)w_{x x}^2\mathrm{d}x
-\kappa\mu\int_{\mathbb{R}}(z_x+\bar u_x)w_{x}w_{x x}\mathrm{d}x\\
&~~-\kappa\mu\int_{\mathbb{R}}z_x\bar\rho_xw_{x x}\mathrm{d}x
-\kappa\mu\int_{\mathbb{R}}z\bar \rho_{xx} w_{x x}\mathrm{d}x\\
=&(a+b)\mu\int_{\mathbb{R}}w_{xx}z_{x x}\mathrm{d}x+\sum_{i=6}^{9}I_{i}.
\end{split}
\end{equation}
Similar to the treatment of \eqref{2.3.10}, we have
\begin{equation}\label{2.3.23}
\begin{split}
I_6&\leq C\int_{\mathbb R}(\|z\|_{L^{\infty}}+\|\bar u\|_{L^{\infty}})w_{x x}^2\mathrm{d}x\\
 &\leq C(\varepsilon_0+\delta)\int_{\mathbb R}w_{x x}^2\mathrm{d}x.
\end{split}
\end{equation}
Recall that $|\bar{u}_{x}| \leq C \delta (1+t)^{-\frac{1}{2}}$, from $\eqref{2.1.1}_2$, \eqref{2.3.3} and Young inequality, it is easy to derive that
\begin{equation}\label{2.3.24}
\begin{split}
 I_7+I_8 &=-\kappa\mu\int_{\mathbb{R}}(z_x+\bar u_x)w_{x}w_{x x}\mathrm{d}x
-\kappa\mu\int_{\mathbb{R}}z_x\bar\rho_xw_{x x}\mathrm{d}x\\
 &\leq \eta\int_{\mathbb R}w_{x x}^2\mathrm{d}x
 +C_\eta\int_{\mathbb R}(\|z_x\|^2_{L^{\infty}}+\|\bar u_x\|^2_{L^{\infty}})w_{x}^2\mathrm{d}x
 +C_\eta\int_{\mathbb R}\|\bar \rho_x\|^2_{L^{\infty}}z_{x}^2\mathrm{d}x \\
 &\leq \eta\int_{\mathbb R}w_{x x}^2\mathrm{d}x
 +C_\eta[\varepsilon_0^2(1+t)^{-\frac{3}{2}}+\delta^2(1+t)^{-1}]\int_{\mathbb R}w_{x}^2\mathrm{d}x
 +C_\eta\delta^2(1+t)^{-1}\int_{\mathbb R}z_{x}^2\mathrm{d}x\\
 &\leq \eta\int_{\mathbb R}w_{x x}^2\mathrm{d}x
 +C_\eta(\varepsilon_0^2+\delta^2)(1+t)^{-1}\int_{\mathbb R}(w_{x}^2+z_{x}^2)\mathrm{d}x.
\end{split}
\end{equation}
Similar to the calculation of \eqref{2.3.11}, we obtain
\begin{equation}\label{2.3.25}
\begin{split}
 I_9 &\leq \eta\int_{\mathbb R}w_{x x}^2\mathrm{d}x+C_\eta\delta^2(1+t)^{-1}\int_{\mathbb R}z^2\widetilde\omega^2\mathrm{d}x.
\end{split}
\end{equation}
Putting \eqref{2.3.23}-\eqref{2.3.25} into \eqref{2.3.22}, then choosing $\eta$ suitably small, we can conclude
\begin{equation}\label{2.3.26}
\begin{split}
&\frac{\rm d}{{\rm d}t}\int_{\mathbb{R}} \left(\frac{\lambda w_x^2}{2}-\mu w_xz_x\right) \mathrm{d}x
+\frac{b\lambda}{4}\int_{\mathbb{R}}w_{xx}^2\mathrm{d}x
+\frac{1}{2}\int_{\mathbb{R}}{(\lambda w_x-\mu z_x)^2}\mathrm{d}x\\
\leq& \frac{1}{2}\int_{\mathbb{R}} (\bar\rho_{xt}
    -b\bar\rho_{xxx})^2\mathrm{d}x
   +C(\varepsilon_0+\delta)(1+t)^{-1}\int_{\mathbb{R}}(w_x^2+z_x^2)\mathrm{d}x\\
   &+C\delta(1+t)^{-1}\int_{\mathbb{R}}z^2 \widetilde\omega^2\mathrm{d}x
   +C\int_{\mathbb{R}}z_{xx}^2\mathrm{d}x.
\end{split}
\end{equation}
Now we only need to estimate the last term on the right-hand side of \eqref{2.3.26}.

Multiplying $\eqref{2.3.17}_1$ by $z_x$, integrating it with respect to $x$, we obtain
\begin{equation}\label{2.3.27}
\begin{split}
\frac{\rm d}{{\rm d}t} \int_{\mathbb{R}}\frac{ z_x^2}{2}\mathrm{d}x +a\int_{\mathbb{R}}z_{xx}^2\mathrm{d}x=&\kappa\int_{\mathbb{R}}(z+\bar u)w_{x x}z_{xx}\mathrm{d}x
+\kappa\int_{\mathbb{R}}(z_x+\bar u_x)w_{x}z_{x x}\mathrm{d}x\\
&+\kappa\int_{\mathbb{R}}z_x\bar\rho_xz_{x x}\mathrm{d}x
+\kappa\int_{\mathbb{R}}z\bar \rho_{xx}z_{x x}\mathrm{d}x.
\end{split}
\end{equation}
Similar to the treatment of \eqref{2.3.23}, \eqref{2.3.24} and \eqref{2.3.25}, one has
\begin{equation}\label{2.3.28}
\begin{split}
\frac{\rm d}{{\rm d}t} \int_{\mathbb{R}}\frac{ z_x^2}{2}\mathrm{d}x +\frac{a}{2}\int_{\mathbb{R}}z_{xx}^2\mathrm{d}x
\leq& C(\varepsilon_0+\delta)\int_{\mathbb{R}}w_{xx}^2\mathrm{d}x+C(\varepsilon_0+\delta)(1+t)^{-1}\int_{\mathbb{R}}(w_x^2+z_x^2)\mathrm{d}x\\
   &+C\delta(1+t)^{-1}\int_{\mathbb{R}}z^2 \widetilde\omega^2\mathrm{d}x.
\end{split}
\end{equation}
Multiplying \eqref{2.3.28} by a big positive constant $K$ and summing it to \eqref{2.3.26}, then using Lemma \ref{Lemma 2.1}, we obtain
\begin{equation}\label{2.3.29}
\begin{split}
&\frac{\rm d}{{\rm d}t} \int_{\mathbb{R}}
   \left(\frac{\lambda w_x^2}{2} +\frac{K z_x^2}{2}-\mu w_xz_x\right) \mathrm{d}x
   +\frac{b\lambda}{8}\int_{\mathbb{R}}w_{xx}^2\mathrm{d}x+\frac{K a}{4}\int_{\mathbb{R}}z_{xx}^2\mathrm{d}x
   +\frac{1}{2}\int_{\mathbb{R}}{(\lambda w_x-\mu z_x)^2}\mathrm{d}x\\
   \leq
&C\delta(1+t)^{-\frac{5}{2}}
    +C(\varepsilon_0+\delta)(1+t)^{-1}\int_{\mathbb{R}}(w_x^2+z_x^2)\mathrm{d}x
    +C\delta(1+t)^{-1}\int_{\mathbb{R}} z^2 \widetilde\omega^2\mathrm{d}x.
\end{split}
\end{equation}
Integrating \eqref{2.3.29} over $(0,t)$ and taking $h=z$ in \eqref{2.2.17}, together with \eqref{2.3.4}, we get
\begin{equation}\label{2.3.30}
\begin{split}
&\|w_{x}(t)\|^{2}+\|z_{x}(t)\|^{2}+\int_{0}^{t}\left(\left\|w_{x x}(\tau)\right\|^{2}+\|z_{x x}(\tau)\|^{2}+\|(\lambda w_x-\mu z_x)(\tau)\|^2\right) \mathrm{d} \tau\\
\leq& C\left(\left\|w_{0}\right\|_{1}^{2}+\left\|z_{0}\right\|_1^{2}+\delta\right).
\end{split}
\end{equation}
Multiplying \eqref{2.3.29} by $(1 + t)$, and integrating it with respect to $t$, then by applying \eqref{2.2.17} and \eqref{2.3.4}, one can immediately obtain \eqref{2.3.16}. The proof of Lemma \ref{Lemma 2.4} is completed.
\end{proof}
\begin{lemma}\label{Lemma 2.5}
If $N(T) \leq \varepsilon_0^2$ and $\delta$ are small enough, it holds that
\begin{equation}\label{2.3.31}
\begin{split}
&(1+t)^{2}(\|w_{x x}(t)\|^{2}+\|z_{x x}(t)\|^{2})\\
&+\int_{0}^{t}(1+\tau)^{2}\left(\left\|w_{x x x}(\tau)\right\|^{2}+\|z_{x x x}(\tau)\|^{2}+\|(\lambda w_{xx}-\mu z_{xx})(\tau)\|^2\right) \mathrm{d} \tau\\
\leq& C\left(\left\|w_{0}\right\|_{2}^{2}+\left\|z_{0}\right\|_{2}^{2}+\delta\right),
\end{split}
\end{equation}
for $0 \leq t \leq T$.
\end{lemma}
\begin{proof}
Similar to Lemma \ref{Lemma 2.3} and Lemma \ref{Lemma 2.4}, from $\int_{\mathbb{R}}\lambda w_{xx}\times\partial_x^2\eqref{2.1.7}_2\mathrm{d}x-\int_{\mathbb{R}}\mu z_{xx}\times\partial_x^2\eqref{2.1.7}_2\mathrm{d}x$, then applying integration by parts and the equation $\partial_x^2\eqref{2.1.7}_1$, we can get
\begin{equation}\label{2.3.32}
\begin{split}
&\frac{\rm d}{{\rm d}t} \int_{\mathbb{R}}
   \left(\frac{\lambda w_{xx}^2}{2}-\mu w_{xx}z_{xx}\right)\mathrm{d}x
   +{b\lambda}\int_{\mathbb{R}}w_{xxx}^2\mathrm{d}x
   +\int_{\mathbb{R}} (\lambda w_{xx}-\mu z_{xx})^2\mathrm{d}x\\
   &~~~+\int_{\mathbb{R}} (\lambda w_{xx}-\mu z_{xx})(\bar\rho_{xxt}-b\bar\rho_{xxxx})\mathrm{d}x\\
&=(a+b)\mu\int_{\mathbb{R}}w_{xxx}z_{xxx}\mathrm{d}x-\kappa\mu\int_{\mathbb{R}}w_{xxx}[(z+\bar u)w_x+z\bar\rho_x]_{xx}\mathrm{d}x\\
&=(a+b)\mu\int_{\mathbb{R}}w_{xxx}z_{xxx}\mathrm{d}x
-\kappa\mu\int_{\mathbb{R}}w_{x x x}w_{x}z_{x x}\mathrm{d}x
-2\kappa\mu\int_{\mathbb{R}}w_{x x x}z_{x}w_{x x}\mathrm{d}x
-\kappa\mu\int_{\mathbb{R}}w_{x x x}\bar u_{x x}w_{x}\mathrm{d}x\\
&~~~-2\kappa\mu\int_{\mathbb{R}}w_{x x x}\bar u_{x}w_{x x}\mathrm{d}x
-\kappa\mu\int_{\mathbb{R}}w_{x x x}\bar \rho_{x}z_{x x}\mathrm{d}x
-2\kappa\mu\int_{\mathbb{R}}w_{x x x}\bar \rho_{x x}z_{x}\mathrm{d}x
-\kappa\mu\int_{\mathbb{R}}(z+\bar u)w_{x x x}^2\mathrm{d}x\\
&~~~-\kappa\mu\int_{\mathbb{R}}w_{x x x}z\bar \rho_{x x x}\mathrm{d}x\\
&=(a+b)\mu\int_{\mathbb{R}}w_{xxx}z_{xxx}\mathrm{d}x+\sum_{i=10}^{17}I_{i}.
\end{split}
\end{equation}
By applying \eqref{2.3.2}-\eqref{2.3.3} and Young inequality, we get
\begin{equation}\label{2.3.33}
 \begin{split}
  I_{10}+I_{11}&\leq\eta\int_{\mathbb R}w_{x x x}^2\mathrm{d}x
 +C_\eta\int_{\mathbb R}(\|w_x\|^2_{L^{\infty}}z_{x x}^2+\|z_x\|^2_{L^{\infty}}w_{x x}^2)\mathrm{d}x\\
 &\leq\eta\int_{\mathbb R}w_{x x x}^2\mathrm{d}x + C_\eta\varepsilon^2_0(1+t)^{-\frac{3}{2}}\int_{\mathbb R}z_{x x}^2\mathrm{d}x+C_\eta\varepsilon^2_0(1+t)^{-\frac{3}{2}}\int_{\mathbb R}w_{x x}^2\mathrm{d}x.
\end{split}
\end{equation}
Noting that $|\bar{u}_{x}| \leq C \delta (1+t)^{-\frac{1}{2}}$ and $|\bar{u}_{x x}| \leq C \delta (1+t)^{-1}$, using $\eqref{2.1.1}_2$ and Young inequality, one yields that
\begin{equation}\label{2.3.34}
 \begin{split}
  &I_{12}+I_{13}+I_{14}+I_{15}\\
\leq&\eta\int_{\mathbb R}w_{x x x}^2\mathrm{d}x
 +C_\eta\int_{\mathbb R}(\|\bar u_{x x}\|^2_{L^{\infty}}w_{x}^2+\|\bar u_{x}\|^2_{L^{\infty}}w_{x x}^2+\|\bar \rho_{x}\|^2_{L^{\infty}}z_{x x}^2+\|\bar \rho_{x x}\|^2_{L^{\infty}}z_{x}^2)\mathrm{d}x\\
 \leq&\eta\int_{\mathbb R}w_{x x x}^2\mathrm{d}x + C_\eta\delta^2(1+t)^{-2}\int_{\mathbb R}w_{x}^2\mathrm{d}x
 +C_\eta \delta^2 (1+t)^{-1}\int_{\mathbb R}w_{x x}^2\mathrm{d}x\\
 &+C_\eta \delta^2 (1+t)^{-1}\int_{\mathbb R}z_{x x}^2\mathrm{d}x+C_\eta \delta^2 (1+t)^{-2}\int_{\mathbb R}z_{x}^2\mathrm{d}x.
     \end{split}
   \end{equation}
Similar to the calculation of \eqref{2.3.10} and \eqref{2.3.11}, we obtain
\begin{equation}\label{2.3.35}
 I_{16}\leq C(\varepsilon_0+\delta)\int_{\mathbb{R}}w_{xxx}^2\mathrm{d}x,
\end{equation}
and
\begin{equation}\label{2.3.36}
 I_{17}\leq\eta \int_{\mathbb{R}}w_{xxx}^2\mathrm{d}x+C_\eta \delta^2(1+t)^{-2}\int_{\mathbb{R}}z^2\widetilde\omega^2\mathrm{d}x.
\end{equation}
Putting \eqref{2.3.33}-\eqref{2.3.36} into \eqref{2.3.32}, by applying Young inequality, and choosing $\eta$ suitably small, we derive that
\begin{equation}\label{2.3.37}
\begin{split}
&\frac{\rm d}{{\rm d}t}\int_{\mathbb{R}} \left(\frac{\lambda w_{xx}^2}{2}-\mu w_{xx}z_{xx}\right) \mathrm{d}x
+\frac{b\lambda}{4}\int_{\mathbb{R}}w_{xxx}^2\mathrm{d}x
+\frac{1}{2}\int_{\mathbb{R}}{(\lambda w_{xx}-\mu z_{xx})^2}\mathrm{d}x\\
\leq& \frac{1}{2}\int_{\mathbb{R}} (\bar\rho_{xxt}
    -b\bar\rho_{xxxx})^2\mathrm{d}x
   +C(\varepsilon_0+\delta)(1+t)^{-1}\int_{\mathbb{R}}(w_{xx}^2+z_{xx}^2)\mathrm{d}x\\
   &+C\delta(1+t)^{-2}\int_{\mathbb{R}}(w_{x}^2+z_{x}^2)\mathrm{d}x
   +C\delta(1+t)^{-2}\int_{\mathbb{R}}z^2 \widetilde\omega^2\mathrm{d}x
   +C\int_{\mathbb{R}}z_{xxx}^2\mathrm{d}x.
\end{split}
\end{equation}
Next, multiplying $\partial_x^2\eqref{2.1.7}_1$ by $K z_{xx}$ ($K$ is sufficiently large) and integrating it with respect to $x$ over $\mathbb{R}$, then similar to the treatment of \eqref{2.3.33}-\eqref{2.3.36}, we can reach
\begin{equation}\label{2.3.38}
\begin{split}
\frac{\rm d}{{\rm d}t} \int_{\mathbb{R}}\frac{ K z_{xx}^2}{2}\mathrm{d}x +\frac{Ka}{2}\int_{\mathbb{R}}z_{xxx}^2\mathrm{d}x
\leq& C(\varepsilon_0+\delta)\int_{\mathbb{R}}w_{xxx}^2\mathrm{d}x+C(\varepsilon_0+\delta)(1+t)^{-1}\int_{\mathbb{R}}(w_{xx}^2+z_{xx}^2)\mathrm{d}x\\
   &+C\delta(1+t)^{-2}\int_{\mathbb{R}}(w_{x}^2+z_{x}^2)\mathrm{d}x+C\delta(1+t)^{-2}\int_{\mathbb{R}}z^2 \widetilde\omega^2\mathrm{d}x.
\end{split}
\end{equation}
Summing \eqref{2.3.38} with \eqref{2.3.37}, and using Lemma \ref{Lemma 2.1}, we can obtain
\begin{equation}\label{2.3.39}
\begin{split}
&\frac{\rm d}{{\rm d}t} \int_{\mathbb{R}}
   \left(\frac{\lambda{w_{xx}}^2}{2} +\frac{K z_{xx}^2}{2}-\mu w_{xx}z_{xx}\right) \mathrm{d}x
   +\frac{b\lambda}{8}\int_{\mathbb{R}}w_{xxx}^2\mathrm{d}x+\frac{K a}{4}\int_{\mathbb{R}}z_{xxx}^2\mathrm{d}x\\
   &~+\frac{1}{2}\int_{\mathbb{R}}{(\lambda w_{xx}-\mu z_{xx})^2}\mathrm{d}x\\
\leq&C\delta(1+t)^{-\frac{7}{2}}
    +C(\varepsilon_0+\delta)(1+t)^{-1}\int_{\mathbb{R}}(w_{xx}^2+z_{xx}^2)\mathrm{d}x
    +C\delta(1+t)^{-2}\int_{\mathbb{R}}(w_{x}^2+z_{x}^2)\mathrm{d}x\\
    &~+C\delta(1+t)^{-2}\int_{\mathbb{R}} z^2 \widetilde\omega^2\mathrm{d}x.
\end{split}
\end{equation}
Integrating \eqref{2.3.39} over $(0,t)$, then taking $h =z$ in \eqref{2.2.17}, together with using \eqref{2.3.4} and \eqref{2.3.16}, we reach
\begin{equation}\label{2.3.40}
\begin{split}
&\|w_{x x}(t)\|^{2}+\|z_{x x}(t)\|^{2}+\int_{0}^{t}\left(\left\|w_{x x x}(\tau)\right\|^{2}+\|z_{x x x}(\tau)\|^{2}+\|(\lambda w_{xx}-\mu z_{xx})(\tau)\|^2\right) \mathrm{d} \tau \\
\leq& C\left(\left\|w_{0}\right\|_{2}^{2}+\left\|z_{0}\right\|_{2}^{2}+\delta\right).
\end{split}
\end{equation}
Multiplying \eqref{2.3.39} by $(1 + t)^{2}$, integrating with respect to $t$, then using \eqref{2.2.17}, \eqref{2.3.4} and \eqref{2.3.16}, we can immediately obtain \eqref{2.3.31}. The proof of Lemma \ref{Lemma 2.5} is completed.
\end{proof}
\begin{lemma}\label{Lemma 2.6}
If $N(T) \leq \varepsilon_0^2$ and $\delta$ are small enough, it holds that
\begin{equation}\label{2.3.41}
\begin{split}
&(1+t)^{2}(\|w_{t}(t)\|^{2}+\|z_{t}(t)\|^{2})\\
&+\int_{0}^{t}(1+\tau)^{2}\left(\left\|w_{xt}(\tau)\right\|^{2}+\|z_{xt}(\tau)\|^{2}+\|(\lambda w_{t}-\mu z_{t})(\tau)\|^2\right) \mathrm{d} \tau\\
\leq& C\left(\left\|w_{0}\right\|_{2}^{2}+\left\|z_{0}\right\|_{2}^{2}+\delta\right),
\end{split}
\end{equation}
for $0 \leq t \leq T$.
\end{lemma}
\begin{proof}
Firstly, having $\int_{\mathbb{R}}\lambda w_{t}\times\partial_t\eqref{2.1.7}_2\mathrm{d}x-\int_{\mathbb{R}}\mu z_{t}\times\partial_t\eqref{2.1.7}_2\mathrm{d}x$, then applying integration by parts and the equation $\partial_t\eqref{2.1.7}_1$, we can get
\begin{equation}\label{2.3.42}
\begin{split}
&\frac{\rm d}{{\rm d}t} \int_{\mathbb{R}}
   \left(\frac{\lambda w_{t}^2}{2}-\mu w_{t}z_{t}\right)\mathrm{d}x
   +{b\lambda}\int_{\mathbb{R}}w_{xt}^2\mathrm{d}x
   +\int_{\mathbb{R}} (\lambda w_{t}-\mu z_{t})^2\mathrm{d}x\\
   &~~~+\int_{\mathbb{R}} (\lambda w_{t}-\mu z_{t})(\bar\rho_{tt}-b\bar\rho_{xxt})\mathrm{d}x\\
&=(a+b)\mu\int_{\mathbb{R}}w_{xt}z_{xt}\mathrm{d}x-\kappa\mu\int_{\mathbb{R}}w_{xt}[(z+\bar u)w_x+z\bar\rho_x]_t\mathrm{d}x\\
&=(a+b)\mu\int_{\mathbb{R}}w_{xt}z_{xt}\mathrm{d}x
-\kappa\mu\int_{\mathbb{R}}w_{xt}w_{x}z_{t}\mathrm{d}x
-\kappa\mu\int_{\mathbb{R}}w_{xt}\bar u_{t}w_{x}\mathrm{d}x
-\kappa\mu\int_{\mathbb{R}}w_{xt}\bar \rho_{x}z_{t}\mathrm{d}x\\
&~~~~-\kappa\mu\int_{\mathbb{R}}w_{xt}z\bar \rho_{x t}\mathrm{d}x
-\kappa\mu\int_{\mathbb{R}}(z+\bar u)w_{xt}^2\mathrm{d}x\\
&=(a+b)\mu\int_{\mathbb{R}}w_{xt}z_{xt}\mathrm{d}x+\sum_{i=18}^{22}I_{i}.
\end{split}
\end{equation}
Noting that $|\bar{u}_{t}| \leq C \delta (1+t)^{-1}$, using $\eqref{2.1.1}_2$ and \eqref{2.3.2}, we have
\begin{equation}\label{2.3.43}
 \begin{split}
  &I_{18}+I_{19}+I_{20}\\
\leq&\eta\int_{\mathbb R}w_{xt}^2\mathrm{d}x
 +C_\eta\int_{\mathbb R}\|w_{x}\|^2_{L^{\infty}}z_{t}^2\mathrm{d}x +C_\eta\int_{\mathbb R}\|\bar u_{t}\|^2_{L^{\infty}}w_{x}^2\mathrm{d}x+C_\eta\int_{\mathbb R}\|\bar \rho_{x}\|^2_{L^{\infty}}z_{t}^2\mathrm{d}x\\
\leq&\eta\int_{\mathbb R}w_{xt}^2\mathrm{d}x
 +C_\eta\varepsilon_0^2(1+t)^{-\frac{3}{2}}\int_{\mathbb R}z_{t}^2\mathrm{d}x +C_\eta\delta^2(1+t)^{-2}\int_{\mathbb R}w_{x}^2\mathrm{d}x+C_\eta\delta^2(1+t)^{-1}\int_{\mathbb R}z_{t}^2\mathrm{d}x\\
\leq&\eta\int_{\mathbb R}w_{xt}^2\mathrm{d}x
 +C_\eta(\varepsilon_0^2+\delta^2)(1+t)^{-1}\int_{\mathbb R}z_{t}^2\mathrm{d}x +C_\eta\delta^2(1+t)^{-2}\int_{\mathbb R}w_{x}^2\mathrm{d}x.
\end{split}
\end{equation}
Similar to the calculation of \eqref{2.3.10} and \eqref{2.3.11}, we obtain
\begin{equation}\label{2.3.44}
I_{21}\leq\eta \int_{\mathbb{R}}w_{xt}^2\mathrm{d}x+C_\eta \delta^2(1+t)^{-2}\int_{\mathbb{R}}z^2\widetilde\omega^2\mathrm{d}x,
\end{equation}
and
\begin{equation}\label{2.3.45}
 I_{22}\leq C(\varepsilon_0+\delta)\int_{\mathbb{R}}w_{xt}^2\mathrm{d}x.
\end{equation}
Putting \eqref{2.3.43}-\eqref{2.3.45} into \eqref{2.3.42}, and choosing $\eta$ suitably small, we can conclude that
\begin{equation}\label{2.3.46}
\begin{split}
&\frac{\rm d}{{\rm d}t} \int_{\mathbb{R}}
   \left(\frac{\lambda w_{t}^2}{2}-\mu w_{t}z_{t}\right)\mathrm{d}x
   +\frac{b\lambda}{4}\int_{\mathbb{R}}w_{xt}^2\mathrm{d}x
   +\frac{1}{2}\int_{\mathbb{R}} (\lambda w_{t}-\mu z_{t})^2\mathrm{d}x\\
\leq&\frac{1}{2}\int_{\mathbb{R}} (\bar\rho_{tt}
    -b\bar\rho_{xxt})^2\mathrm{d}x
+C(\varepsilon_0+\delta)(1+t)^{-1}\int_{\mathbb{R}}z_t^2\mathrm{d}x+C\delta(1+t)^{-2}\int_{\mathbb{R}}w_x^2\mathrm{d}x\\
&+C\delta(1+t)^{-2}\int_{\mathbb{R}}z^2\widetilde\omega^2\mathrm{d}x+C\int_{\mathbb{R}}z_{xt}^2\mathrm{d}x.
\end{split}
\end{equation}
Similarly, we can get from $\int_{\mathbb{R}}K z_t\times\partial_t\eqref{2.1.7}_1\mathrm{d}x$ ($K$ is sufficiently large) that
\begin{equation}\label{2.3.47}
\begin{split}
\frac{\rm d}{{\rm d}t} \int_{\mathbb{R}}\frac{ K z_{t}^2}{2}\mathrm{d}x +\frac{K a}{2}\int_{\mathbb{R}}z_{xt}^2\mathrm{d}x
\leq& C(\varepsilon_0+\delta)\int_{\mathbb{R}}w_{xt}^2\mathrm{d}x+C(\varepsilon_0+\delta)(1+t)^{-1}\int_{\mathbb{R}}z_{t}^2\mathrm{d}x\\
   &+C\delta(1+t)^{-2}\int_{\mathbb{R}}w_{x}^2\mathrm{d}x+C\delta(1+t)^{-2}\int_{\mathbb{R}}z^2 \widetilde\omega^2\mathrm{d}x.
\end{split}
\end{equation}
Summing \eqref{2.3.47} to \eqref{2.3.46}, and using Lemma \ref{Lemma 2.1}, we can get
\begin{equation}\label{2.3.48}
\begin{split}
&\frac{\rm d}{{\rm d}t} \int_{\mathbb{R}}
   \left(\frac{\lambda{w_{t}}^2}{2} +\frac{K z_{t}^2}{2}-\mu w_{t}z_{t}\right) \mathrm{d}x
   +\frac{b\lambda}{8}\int_{\mathbb{R}}w_{xt}^2\mathrm{d}x+\frac{K a}{4}\int_{\mathbb{R}}z_{xt}^2\mathrm{d}x+\frac{1}{2}\int_{\mathbb{R}}{(\lambda w_{t}-\mu z_{t})^2}\mathrm{d}x\\
\leq&C\delta(1+t)^{-\frac{7}{2}}
    +C\delta(1+t)^{-2}\int_{\mathbb{R}}w_{x}^2\mathrm{d}x+C\delta(1+t)^{-2}\int_{\mathbb{R}} z^2 \widetilde\omega^2\mathrm{d}x\\
    &+C(\varepsilon_0+\delta)(1+t)^{-1}\int_{\mathbb{R}}z_{t}^2\mathrm{d}x.
\end{split}
\end{equation}
Now we only need to estimate the last term on the right-hand side of \eqref{2.3.48}. By using the equation $\eqref{2.1.7}_1$ and Lemma \ref{Lemma 2.1}, together with \eqref{2.3.3}, it is direct to derive that
\begin{equation}\label{2.3.49}
\int_{\mathbb{R}}z_t^2\mathrm{d}x
\leq C\int_{\mathbb{R}}(w_{xx}^2+z_{xx}^2)\mathrm{d}x+C(1+t)^{-1}\int_{\mathbb{R}}(w_{x}^2+z_{x}^2)\mathrm{d}x+C\delta(1+t)^{-1}\int_{\mathbb{R}}z^2\widetilde\omega^2\mathrm{d}x.
\end{equation}
Noting that $(\lambda w_{t})^2\leq 2(\lambda w_t-\mu z_t)^2+2(\mu z_t)^2$, it follows from \eqref{2.3.48} and \eqref{2.3.49} that
\begin{equation}\label{2.3.50}
\begin{split}
&\frac{\rm d}{{\rm d}t} \int_{\mathbb{R}}
   \left(\frac{\lambda{w_{t}}^2}{2} +\frac{K z_{t}^2}{2}-\mu w_{t}z_{t}\right) \mathrm{d}x
   +\frac{b\lambda}{8}\int_{\mathbb{R}}w_{xt}^2\mathrm{d}x+\frac{K a}{4}\int_{\mathbb{R}}z_{xt}^2\mathrm{d}x+\frac{\lambda^2}{4}\int_{\mathbb{R}}w_t^2\mathrm{d}x+\frac{1}{2}\int_{\mathbb{R}}z_t^2\mathrm{d}x\\
\leq&C\delta(1+t)^{-\frac{7}{2}}
    +C\int_{\mathbb{R}}(w_{xx}^2+z_{xx}^2)\mathrm{d}x+C(1+t)^{-1}\int_{\mathbb{R}}(w_{x}^2+z_{x}^2)\mathrm{d}x+C\delta(1+t)^{-1}\int_{\mathbb{R}}z^2\widetilde\omega^2\mathrm{d}x.
\end{split}
\end{equation}
Integrating \eqref{2.3.50} with respect to $t$, then employing \eqref{2.2.17}, \eqref{2.3.4} and \eqref{2.3.16}, we obtain
\begin{equation}\label{2.3.51}
\begin{split}
&\|w_{t}(t)\|^{2}+\|z_{t}(t)\|^{2}+\int_{0}^{t}\left(\left\|w_{xt}(\tau)\right\|^{2}+\|z_{xt}(\tau)\|^{2}+\|w_{t}(\tau)\|^2+\|z_{t}(\tau)\|^2\right) \mathrm{d} \tau\\
\leq& C\left(\left\|w_{0}\right\|_{2}^{2}+\left\|z_{0}\right\|_{2}^{2}+\delta\right).
\end{split}
\end{equation}
Multiplying \eqref{2.3.50} by $(1 + t)$, we integrate it to obtain
\begin{equation}\label{2.3.52}
\begin{split}
&(1+t)(\|w_{t}(t)\|^{2}+\|z_{t}(t)\|^{2})+\int_{0}^{t}(1+\tau)\left(\left\|w_{xt}(\tau)\right\|^{2}+\|z_{xt}(\tau)\|^{2}+\|w_{t}(\tau)\|^2+\|z_{t}(\tau)\|^2\right) \mathrm{d} \tau\\
\leq& C\left(\left\|w_{0}\right\|_{2}^{2}+\left\|z_{0}\right\|_{2}^{2}+\delta\right).
\end{split}
\end{equation}
At last, multiplying \eqref{2.3.48} by $(1+t)^2$, then using \eqref{2.2.17}, \eqref{2.3.4} and \eqref{2.3.52}, we can conclude \eqref{2.3.41}. The proof of Lemma \ref{Lemma 2.6} is completed.
\end{proof}
Combining Lemma \ref{Lemma 2.3}-Lemma \ref{Lemma 2.6}, one can verify that {\it a priori} assumption $N(T)\leq \varepsilon_0^2 \ll 1$ is closed. In fact, under the {\it a priori} assumption, it is easy to infer that \eqref{2.1.9} holds, the assumption $N(T)\leq \varepsilon_0^2 \ll 1$ always holds provided $\delta$ and $\varepsilon_0$ are sufficiently small. The global existence of the solutions to the Cauchy problem \eqref{2.1.7}-\eqref{2.1.8} follows from the standard continuation argument based on the local existence and the {\it a priori} estimates. The proof of Theorem \ref{Thm 2.1} is completed.

\section{Initial-boundary value problem}\label{S3}
{\numberwithin{equation}{subsection}

\subsection{The case of Dirichlet boundary condition}\label{S3.1}
\subsubsection{Reformulation of the problem and theorem}\label{S3.1.1}
In this subsection, we consider the problem \eqref{1.2} and \eqref{1.6} with the Dirichlet boundary condition \eqref{1.7}. Motivated by \cite{Nishihara-Yang1999}, in the case of $u_+\neq u_-$, without loss of generality, we assume $u_-<u_+$, putting $\rho_t=0$ and $-b\rho_{xx}=0$ in $\eqref{1.2}_2$, we have $u_t-au_{xx}+\kappa[u\rho_x]_x=0$ and $\lambda\rho-\mu u=0$. To construct the diffusion waves $(\bar u,\bar \rho)(x,t)$, it is known that we have a self-similar solution $\tau=\phi(x/{\sqrt{1+t}})$ satisfying
\begin{equation}\notag
\left\{\begin{array}{l}
\tau_t=\left[\left(a-\frac{\kappa\mu}{\lambda}\tau\right)\tau_x\right]_x,\quad (x,t)\in\mathbb{R}\times\mathbb{R}^+,\\[2mm]
\tau|_{x=\pm\infty}=u_{\pm},
 \end{array}
        \right.
\end{equation}
for any constant $u_->0$. Therefore, for $u_-<\beta< u_+$, there exists a unique $\bar u(x,t)$ in the form of $\phi(x/{\sqrt{1+t}})|_{x\geq 0}$ satisfying
\begin{equation}\label{3.1.1}
\left\{\begin{array}{l}
\bar u_t
   -
   [(a-\frac{\kappa\mu}{\lambda}\bar u)\bar u_x]_x=0,\\[2mm]
\bar u|_{x=0}=\beta, \quad \bar u|_{x=+\infty}=u_{+}.\\
 \end{array}
        \right.
\end{equation}
Defining the perturbation as \eqref{2.1.6}, we have the reformulated problem
\begin{equation}\label{3.1.2}
\left\{\begin{array}{l}
z_t
   -az_{xx}
   +\kappa[(z+\bar{u})w_x+z\bar \rho_x]_x=0,\\[2mm]
w_t
    -bw_{xx}
    +\lambda w
    -\mu z+\bar\rho_t
    -b\bar\rho_{xx}=0,\\[2mm]
 \end{array}
        \right.
\end{equation}
with the initial-boundary data
\begin{equation}\label{3.1.3}
\left\{\begin{array}{l}
\left(w,z)\right|_{t=0}=\left(w_{0},z_{0}\right)(x) \rightarrow 0\quad
 \text {as}
  \quad
  x\rightarrow \infty,\\[2mm]
\left(w,z)\right|_{x=0}=(0,0).
 \end{array}
        \right.
\end{equation}
\begin{theorem}\label{Thm 3.1} {\bf (Dirichlet boundary).}
Suppose that $u_-<u_+$ and $\delta :=|u_+|+|u_-|$. There exists a positive constant $\varepsilon_0$ such that if $\delta+\|w_0\|_2+\|z_0\|_2 < \varepsilon_0$, then the initial-boundary value problem \eqref{3.1.2}-\eqref{3.1.3} admits a unique time-global solution $(w,z)(x, t)$, which satisfies
\begin{equation}\notag
w \in C^{i, \infty}\left([0, \infty) ; H^{2-i}\right), \quad i=0,1,2,~~~
z \in C^{i, \infty}\left([0, \infty) ; H^{2-i}\right), \quad i=0,1,2,
\end{equation}
and
\begin{equation}\label{3.1.4}
\begin{split}
&\left\|w(t)\right\|_{2}^{2}+\left\|z(t)\right\|_2^{2}+\|w_t(t)\|^2+\|z_t(t)\|^2\\
&+\int_0^t\bigg(\|w_x(\tau)\|_1^2+\|z_x(\tau)\|_1^2+\|(\lambda w-\mu z)(\tau)\|_1^2+\|w_{xt}(\tau)\|^2+\|z_{xt}(\tau)\|^2\\
&+\|(\lambda w_t-\mu z_t)(\tau)\|^2\bigg)\mathrm{d}\tau\\
\leq& C\left(\left\|w_{0}\right\|_{2}^{2}+\left\|z_{0}\right\|_{2}^{2}+\delta\right).
\end{split}
\end{equation}
Moreover, we have
\begin{equation}\label{3.1.5}
\lim _{t \rightarrow+\infty} \sup _{x \in \mathbb{R}^+}\|(w,z)(x, t)\|_1=0,
\end{equation}
or
\begin{equation}\label{3.1.6}
\lim _{t \rightarrow+\infty} \sup _{x \in \mathbb{R}^+}\|(\rho-\bar \rho, u-\bar u)(x, t)\|_1=0,
\end{equation}
that is to say that the solution $(u,\rho)(x,t)$ to the initial-boundary value problem \eqref{1.2}, \eqref{1.6} and \eqref{1.7} tends time-asymptotically to the diffusion waves.
\end{theorem}
\begin{remark}
 As one can see from Theorem \ref{Thm 3.1},  the solution $(w,z)(x, t)$ has no decay rate. The main reason is that we encounter difficulties in dealing with some boundary terms, and finally use the structure of the reformulated equations to close the {\it a priori} assumption.
\end{remark}
\subsubsection{Preliminaries}\label{S3.1.2}
In this subsection, we will give some fundamental dissipative properties of nonlinear diffusion waves $(\bar u,\bar \rho)(x,t)$ on the half line $\mathbb{R}^+$ which is similar to Lemma \ref{Lemma 2.1}, and some inequalities concerning the heat kernel on the half line $\mathbb{R}^+$.
\begin{lemma}\label{Lemma 3.1}
For each $p\in [1,\infty]$ is an integer and $u_-\leq u_{+},$  it is easy to verify that
\begin{equation}\label{3.1.7}
\begin{split}
&u_-\leq\beta\leq \bar u(x,t)\leq u_{+}, \\
&\|\partial_{t}^{l} \partial_{x}^{k}\bar u(x,t)\|_{L^{p}}\leq C|u_+-u_-|(1+t)^{-\frac{k}{2}-l+\frac{1}{2p}}, \quad k,l\geq0, ~~k+l\geq 1,\\
&\|\partial_{t}^{l} \partial_{x}^{k}\bar \rho(x,t)\|_{L^{p}}\leq C|u_+-u_-|(1+t)^{-\frac{k}{2}-l+\frac{1}{2p}}, ~\quad k,l\geq0, ~~k+l\geq 1.
\end{split}
\end{equation}
\end{lemma}
For the inequalities concerning the heat kernel on the half line $\mathbb{R}^+$, we only need to define
\begin{equation}\label{3.1.8}
\widetilde\omega(x, t)=(1+t)^{-\frac{1}{2}} \exp \left\{-\frac{\alpha x^{2}}{1+t}\right\},
\quad g(x, t)=\int_{0}^{x} \widetilde\omega(y, t) \mathrm{d} y.
\end{equation}
It is easy to check that
\begin{equation}\notag
4 \alpha g_{t}=\widetilde\omega_{x}, \quad\|g(\cdot, t)\|_{L^{\infty}}=\frac{1}{2}\sqrt{\pi} \alpha^{-\frac{1}{2}}.
\end{equation}
Similar to the calculation of \eqref{2.2.11} and \eqref{2.2.17}, we can get the inequalities in Lemma \ref{Lemma 3.2} and Corollary \ref{Corollary 3.1}, the details are omitted.
\begin{lemma}\label{Lemma 3.2}
For $0<T \leq+\infty,$ assume that $h(x, t)$ satisfies
$$
h_{x} \in L^{2}\left(0,T: L^{2}(\mathbb{R}^+)\right), \quad h_{t} \in L^{2}\left(0, T: H^{-1}(\mathbb{R}^+)\right).
$$
Then the following estimate holds:
\begin{equation}\label{3.1.9}
\begin{split}
&\int_0^T \int_{\mathbb{R}^+} h^{2} \widetilde\omega^{2} \mathrm{d}x \mathrm{d}t \\
\leq&
 \pi\|h(0)\|^{2}+\pi \alpha^{-1} \int_{0}^{T}\left\|h_{x}(t)\right\|^{2} \mathrm{d}t+8 \alpha \int_{0}^{T}\left\langle h_{t}, h g^{2}\right\rangle_{H^{-1} \times H^{1}} \mathrm{d}t,
\end{split}
\end{equation}
where $\langle \cdot,\cdot \rangle$ denotes the inner product on $H^{-1}(\mathbb{R}^+) \times H^{1}(\mathbb{R}^+)$.
\end{lemma}
\begin{corollary}\label{Corollary 3.1}
In addition to the condition of Lemma \ref{Lemma 3.2}, assume that $\delta :=|u_+|+|u_-|\ll 1$, $\|h\|_{L^{\infty}(\mathbb{R}^+)} \leq \varepsilon\ll 1$ and $h$ satisfies
\begin{equation}\label{3.1.10}
h_t=
   ah_{xx}
   -\kappa[(h+\bar{u})w_x+h\bar \rho_x]_x,\quad
   h(x,0)=h_{0}(x) \in L^2(\mathbb{R}^+),\quad h_x(+\infty,t)=0,
\end{equation}
where $a$ and $\kappa$ are given positive constant, $\bar{u}$ is the self-similar solutions of \eqref{3.1.1} and $\bar{\rho}=\frac{\mu}{\lambda}\bar{u}$. Then there exists some positive constant $C$ such that
\begin{equation}\label{3.1.11}
\int_0^T\int_{\mathbb {R}^+ }h^2 \widetilde\omega^2\mathrm{d}x \mathrm{d}t
\leq
C\int_0^T(\|h_x(\tau)\|^2+\|w_x(\tau)\|^2)\mathrm{d}\tau
+C\|h_0\|^2.
\end{equation}
\end{corollary}
Finally, we show the following lemma (see \cite{{Matsumura-Nishihara1986}}).
\begin{lemma}\label{Lemma 3.3}
If $g(t) \geq 0$, $g(t) \in L^{1}(0,\infty)$ and $g^{\prime}(t) \in L^{1}(0, \infty)$, then $g(t) \rightarrow 0$ as $t \rightarrow \infty$.
\end{lemma}
\subsubsection{Proof of Theorem \ref{Thm 3.1}}\label{S3.1.3}
In this subsection, we devote ourselves to the proof of Theorem \ref{Thm 3.1}. It is well known that the global existence can be obtained by the continuation argument based on the local existence of solutions and {\it a priori} estimates. The local existence of \eqref{3.1.2} and \eqref{3.1.3} can be easily derived by using the standard method and its proof is omitted for brevity. In the following, our main effort will be to prove the {\it a priori} estimates of the solution $(w,z)(x,t)$ under the {\it a priori} assumption
\begin{equation}\label{3.1.12}
N(T):=\sup _{0\leq t\leq T}\left(\|w\|_2^2+\|z\|_2^2\right)\leq \varepsilon_0^2,
\end{equation}
where $0<\varepsilon_0\ll 1$.

Also, from \eqref{3.1.2}, $z|_{x=0}=0$ and $w|_{x=0}=0$ give the following boundary conditions:
\begin{equation}\label{3.1.13}
z(0,t)=z_t(0,t)=w(0,t)=w_t(0,t)=0,
\end{equation}
\begin{equation}\label{3.1.14}
|w_{xx}(0,t)|=|(\bar \rho_t-b\bar \rho_{xx})(0,t)|=|-b\bar \rho_{xx}(0,t)|\leq\|-b\bar \rho_{xx}\|_{L^{\infty}}\leq C\delta(1+t)^{-1}.
\end{equation}
With \eqref{3.1.13}-\eqref{3.1.14} in hand, we now turn to prove Theorem \ref{Thm 3.1}, which will be given by the following series of lemmas.
\begin{lemma}\label{Lemma 3.4}
Under the assumptions of Theorem \ref{Thm 3.1}, we have
\begin{equation}\label{3.1.15}
\|w(t)\|^{2}+\|z(t)\|^{2}+\int_{0}^{t}\left(\left\|w_{x}(\tau)\right\|^{2}+\|z_{x}(\tau)\|^{2}+\|(\lambda w-\mu z)(\tau)\|^2\right) \mathrm{d} \tau \leq C\left(\left\|w_{0}\right\|^{2}+\left\|z_{0}\right\|^{2}+\delta\right),
\end{equation}
for $0 \leq t \leq T$.
\end{lemma}
\begin{proof}
In a similar method as Lemma \ref{Lemma 2.3}, after $\int_{\mathbb{R}^+}\lambda w\times\eqref{3.1.2}_2\mathrm{d}x-\int_{\mathbb{R}^+}\mu z\times\eqref{3.1.2}_2\mathrm{d}x$, applying $\eqref{3.1.2}_1$, the {\it a priori} assumption \eqref{3.1.12} and Lemma \ref{Lemma 3.1}, we can get
\begin{equation}\label{3.1.16}
\begin{split}
&\frac{\rm d}{{\rm d}t}\int_{\mathbb{R}^+} \left(\frac{\lambda w^2}{2}-\mu wz\right) \mathrm{d}x
+\frac{b\lambda}{4}\int_{\mathbb{R}^+}w_x^2\mathrm{d}x
+\frac{1}{2}\int_{\mathbb{R}^+}{(\lambda w-\mu z)^2}\mathrm{d}x\\
\leq& \frac{1}{2}\int_{\mathbb{R}^+} (\bar\rho_t
    -b\bar\rho_{xx})^2\mathrm{d}x
   +C\delta\int_{\mathbb{R}^+} z^2 \widetilde\omega^2\mathrm{d}x+C\int_{\mathbb{R}^+}z_{x}^2\mathrm{d}x.
\end{split}
\end{equation}
Here we have used inequality $\int_{\mathbb R^+}z^2\bar\rho_x^2\mathrm{d}x\leq C\delta^2\int_{\mathbb R^+}z^2\widetilde\omega^2\mathrm{d}x$ and \eqref{3.1.13}.

Next, from $\int_{\mathbb R^+}z\times \eqref{3.1.2}_1\mathrm{d}x$, then similar to the treatment of \eqref{2.3.10} and \eqref{2.3.11}, we have from \eqref{3.1.13} and the {\it a priori} assumption \eqref{3.1.12} that
\begin{equation}\label{3.1.17}
\frac{\rm d}{{\rm d}t} \int_{\mathbb{R}^+}\frac{ z^2}{2}\mathrm{d}x +\frac{a}{2}\int_{\mathbb{R}^+}z_{x}^2\mathrm{d}x
\leq C(\varepsilon_0+\delta)\int_{\mathbb{R}^+}w_x^2\mathrm{d}x+C\delta\int_{\mathbb{R}^+} z^2 \widetilde\omega^2\mathrm{d}x.
\end{equation}
Multiplying \eqref{3.1.17} by a big positive constant $K$ and summing it to \eqref{3.1.16}, we derive
\begin{equation}\label{3.1.18}
\begin{split}
&\frac{\rm d}{{\rm d}t} \int_{\mathbb{R}^+}
   \left(\frac{\lambda w^2}{2} +\frac{K z^2}{2}-\mu wz\right) \mathrm{d}x
   +\frac{b\lambda}{8}\int_{\mathbb{R}^+}w_x^2\mathrm{d}x+\frac{K a}{4}\int_{\mathbb{R}^+}z_x^2\mathrm{d}x
   +\frac{1}{2}\int_{\mathbb{R}^+}{(\lambda w-\mu z)^2}\mathrm{d}x\\
   \leq
&C\delta(1+t)^{-\frac{3}{2}}
    +C\delta\int_{\mathbb{R}^+} z^2 \widetilde\omega^2\mathrm{d}x.
\end{split}
\end{equation}
Integrating the resulting inequality with respect to $t$, then taking $h = z$ in \eqref{3.1.11} leads to \eqref{3.1.15}. The proof of Lemma \ref{Lemma 3.4} is completed.
\end{proof}
\begin{lemma}\label{Lemma 3.5}
Under the assumptions of Theorem \ref{Thm 3.1}, we have
\begin{equation}\label{3.1.19}
\begin{split}
&\|w_{x}(t)\|^{2}+\|z_{x}(t)\|^{2}+\int_{0}^{t}\left(\left\|w_{x x}(\tau)\right\|^{2}+\|z_{x x}(\tau)\|^{2}+\|(\lambda w_x-\mu z_x)(\tau)\|^2\right) \mathrm{d} \tau\\
&\leq C\left(\left\|w_{0}\right\|_{1}^{2}+\left\|z_{0}\right\|_{1}^{2}+\delta\right),
\end{split}
\end{equation}
for $0 \leq t \leq T$.
\end{lemma}
\begin{proof}
In a similar way as Lemma \ref{Lemma 2.4}, from $\int_{\mathbb{R}^+}\lambda w_x\times\partial_x\eqref{3.1.2}_2\mathrm{d}x-\int_{\mathbb{R}^+}\mu z_x\times\partial_x\eqref{3.1.2}_2\mathrm{d}x$, then applying $\partial_x\eqref{3.1.2}_1$, the {\it a priori} assumption \eqref{3.1.12} and Lemma \ref{Lemma 3.1}, we have
\begin{equation}\label{3.1.20}
\begin{split}
&\frac{\rm d}{{\rm d}t}\int_{\mathbb{R}^+} \left(\frac{\lambda w_x^2}{2}-\mu w_xz_x\right) \mathrm{d}x
+\frac{b\lambda}{4}\int_{\mathbb{R}^+}w_{xx}^2\mathrm{d}x
+\frac{1}{2}\int_{\mathbb{R}^+}{(\lambda w_x-\mu z_x)^2}\mathrm{d}x\\
\leq& \frac{1}{2}\int_{\mathbb{R}^+} (\bar\rho_{xt}
    -b\bar\rho_{xxx})^2\mathrm{d}x
   +C(\varepsilon_0+\delta)\int_{\mathbb{R}^+}(w_x^2+z_x^2)\mathrm{d}x
   +C\delta(1+t)^{-1}\int_{\mathbb{R}^+}z^2 \widetilde\omega^2\mathrm{d}x
   +C\int_{\mathbb{R}^+}z_{xx}^2\mathrm{d}x\\
   &+\mu w_x(0,t)z_t(0,t)-bw_{xx}(0,t)(\lambda w_x-\mu z_x)(0,t),
\end{split}
\end{equation}
where we have used inequality $\int_{\mathbb R^+}z^2\bar\rho_{xx}^2\mathrm{d}x\leq C\delta^2(1+t)^{-1}\int_{\mathbb R^+}z^2\widetilde\omega^2\mathrm{d}x$. Since $z_t(0,t)=0$, we only need to estimate the last term on the right-hand side of \eqref{3.1.20}. Employing the Sobolev inequality and \eqref{3.1.14} yields
\begin{equation}\label{3.1.21}
\begin{split}
&-bw_{xx}(0,t)(\lambda w_x-\mu z_x)(0,t)\\
\leq & C\delta(1+t)^{-1}\|(\lambda w_x-\mu z_x)\|_{L^{\infty}}\\
\leq & C\delta(1+t)^{-1}\|(\lambda w_x-\mu z_x)\|^{\frac{1}{2}}\|(\lambda w_{xx}-\mu z_{xx})\|^{\frac{1}{2}}\\
\leq & C\delta(1+t)^{-2}+C\delta\|(\lambda w_x-\mu z_x)\|^2+C\delta\|w_{xx}\|^2+C\delta\|z_{xx}\|^2.
\end{split}
\end{equation}
Substituting \eqref{3.1.21} into \eqref{3.1.20}, one obtains that
\begin{equation}\label{3.1.22}
\begin{split}
&\frac{\rm d}{{\rm d}t}\int_{\mathbb{R}^+} \left(\frac{\lambda w_x^2}{2}-\mu w_xz_x\right) \mathrm{d}x
+\frac{b\lambda}{8}\int_{\mathbb{R}^+}w_{xx}^2\mathrm{d}x
+\frac{1}{4}\int_{\mathbb{R}^+}{(\lambda w_x-\mu z_x)^2}\mathrm{d}x\\
\leq& C\delta(1+t)^{-2}
   +C(\varepsilon_0+\delta)\int_{\mathbb{R}^+}(w_x^2+z_x^2)\mathrm{d}x
   +C\delta(1+t)^{-1}\int_{\mathbb{R}^+}z^2 \widetilde\omega^2\mathrm{d}x+C\int_{\mathbb{R}^+}z_{xx}^2\mathrm{d}x.
\end{split}
\end{equation}
Next, multiplying $\partial_x\eqref{3.1.2}_1$ by $K z_{x}$ ($K$ is sufficiently large) and integrating it with respect to $x$ over $\mathbb{R}^+$, similar to the treatment of $\eqref{2.3.22}$, together with the {\it a priori} assumption \eqref{3.1.12}, we have
\begin{equation}\label{3.1.23}
\begin{split}
\frac{\rm d}{{\rm d}t} \int_{\mathbb{R}^+}\frac{ K z_x^2}{2}\mathrm{d}x +\frac{K a}{2}\int_{\mathbb{R}^+}z_{xx}^2\mathrm{d}x
\leq& C(\varepsilon_0+\delta)\int_{\mathbb{R}^+}w_{xx}^2\mathrm{d}x+C(\varepsilon_0+\delta)\int_{\mathbb{R}^+}(w_x^2+z_x^2)\mathrm{d}x\\
   &+C\delta(1+t)^{-1}\int_{\mathbb{R}^+}z^2 \widetilde\omega^2\mathrm{d}x,
\end{split}
\end{equation}
where the boundary term $Kz_x(0,t)\{-az_{xx}+\kappa[(z+\bar u)w_x+z\bar\rho_x]_x\}(0,t)=-Kz_x(0,t)z_t(0,t)=0$ due to \eqref{3.1.13}.
Thus, combining \eqref{3.1.22} and \eqref{3.1.23}, we get
\begin{equation}\label{3.1.24}
\begin{split}
&\frac{\rm d}{{\rm d}t}\int_{\mathbb{R}^+} \left(\frac{\lambda w_x^2}{2}+\frac{K z_x^2}{2}-\mu w_xz_x\right) \mathrm{d}x
+\frac{b\lambda}{16}\int_{\mathbb{R}^+}w_{xx}^2\mathrm{d}x+\frac{K a}{4}\int_{\mathbb{R}^+}z_{xx}^2\mathrm{d}x
+\frac{1}{4}\int_{\mathbb{R}^+}{(\lambda w_x-\mu z_x)^2}\mathrm{d}x\\
\leq& C\delta(1+t)^{-2}
   +C(\varepsilon_0+\delta)\int_{\mathbb{R}^+}(w_x^2+z_x^2)\mathrm{d}x
   +C\delta(1+t)^{-1}\int_{\mathbb{R}^+}z^2 \widetilde\omega^2\mathrm{d}x.
\end{split}
\end{equation}
Integrating \eqref{3.1.24} with respect to $t$ and taking $h = z$ in \eqref{3.1.11}, together with \eqref{3.1.15}, we get the desired inequality \eqref{3.1.19}. The proof of Lemma \ref{Lemma 3.5} is completed.
\end{proof}
\begin{lemma}\label{Lemma 3.6}
Under the assumptions of Theorem \ref{Thm 3.1}, we have
\begin{equation}\label{3.1.25}
\begin{split}
&(\|w_{t}(t)\|^{2}+\|z_{t}(t)\|^{2})+\int_{0}^{t}\left(\left\|w_{xt}(\tau)\right\|^{2}+\|z_{xt}(\tau)\|^{2}+\|(\lambda w_{t}-\mu z_{t})(\tau)\|^2\right) \mathrm{d} \tau\\
\leq& C\left(\left\|w_{0}\right\|_{2}^{2}+\left\|z_{0}\right\|_{2}^{2}+\delta\right),
\end{split}
\end{equation}
for $0 \leq t \leq T$.
\end{lemma}
\begin{proof}
In a similar way as Lemma \ref{Lemma 2.6}, from $\int_{\mathbb{R}^+}\lambda w_t\times\partial_t\eqref{3.1.2}_2\mathrm{d}x-\int_{\mathbb{R}^+}\mu z_t\times\partial_t\eqref{3.1.2}_2\mathrm{d}x$, then applying $\partial_t\eqref{3.1.2}_1$, the {\it a priori} assumption \eqref{3.1.12} and Lemma \ref{Lemma 3.1}, together with \eqref{3.1.13}, we have
\begin{equation}\label{3.1.26}
\begin{split}
&\frac{\rm d}{{\rm d}t} \int_{\mathbb{R}^+}
   \left(\frac{\lambda w_{t}^2}{2}-\mu w_{t}z_{t}\right)\mathrm{d}x
   +\frac{b\lambda}{4}\int_{\mathbb{R}^+}w_{xt}^2\mathrm{d}x
   +\frac{1}{2}\int_{\mathbb{R}^+} (\lambda w_{t}-\mu z_{t})^2\mathrm{d}x\\
\leq&\frac{1}{2}\int_{\mathbb{R}^+} (\bar\rho_{tt}
    -b\bar\rho_{xxt})^2\mathrm{d}x
+C(\varepsilon_0+\delta)\int_{\mathbb{R}^+}z_t^2\mathrm{d}x+C\delta(1+t)^{-2}\int_{\mathbb{R}^+}w_x^2\mathrm{d}x\\
&+C\delta(1+t)^{-2}\int_{\mathbb{R}^+}z^2\widetilde\omega^2\mathrm{d}x+C\int_{\mathbb{R}^+}z_{xt}^2\mathrm{d}x.
\end{split}
\end{equation}
Next, multiplying $\partial_t\eqref{3.1.2}_1$ by $K z_{t}$ ($K$ is sufficiently large) and integrating it with respect to $x$ over $\mathbb{R}^+$, similar to the treatment of \eqref{2.3.42}, together with the {\it a priori} assumption \eqref{3.1.12} and \eqref{3.1.13}, we have
\begin{equation}\label{3.1.27}
\begin{split}
\frac{\rm d}{{\rm d}t} \int_{\mathbb{R}^+}\frac{ K z_{t}^2}{2}\mathrm{d}x +\frac{K a}{2}\int_{\mathbb{R}^+}z_{xt}^2\mathrm{d}x
\leq& C(\varepsilon_0+\delta)\int_{\mathbb{R}^+}w_{xt}^2\mathrm{d}x+C(\varepsilon_0+\delta)\int_{\mathbb{R}^+}z_{t}^2\mathrm{d}x\\
   &+C\delta(1+t)^{-2}\int_{\mathbb{R}^+}w_{x}^2\mathrm{d}x+C\delta(1+t)^{-2}\int_{\mathbb{R}^+}z^2 \widetilde\omega^2\mathrm{d}x.
\end{split}
\end{equation}
Combining \eqref{3.1.26} and \eqref{3.1.27}, we derive
\begin{equation}\label{3.1.28}
\begin{split}
&\frac{\rm d}{{\rm d}t} \int_{\mathbb{R}^+}
   \left(\frac{\lambda{w_{t}}^2}{2} +\frac{K z_{t}^2}{2}-\mu w_{t}z_{t}\right) \mathrm{d}x
   +\frac{b\lambda}{8}\int_{\mathbb{R}^+}w_{xt}^2\mathrm{d}x+\frac{K a}{4}\int_{\mathbb{R}^+}z_{xt}^2\mathrm{d}x+\frac{1}{2}\int_{\mathbb{R}^+}{(\lambda w_{t}-\mu z_{t})^2}\mathrm{d}x\\
\leq&C\delta(1+t)^{-\frac{7}{2}}
    +C\delta(1+t)^{-2}\int_{\mathbb{R}^+}w_{x}^2\mathrm{d}x+C\delta(1+t)^{-2}\int_{\mathbb{R}^+} z^2 \widetilde\omega^2\mathrm{d}x+C(\varepsilon_0+\delta)\int_{\mathbb{R}^+}z_{t}^2\mathrm{d}x.
\end{split}
\end{equation}
Then by applying the equation $\eqref{3.1.2}_1$, Lemma \ref{Lemma 3.1} and the {\it a priori} assumption \eqref{3.1.12}, we get
\begin{equation}\label{3.1.29}
\int_{\mathbb{R}^+}z_t^2\mathrm{d}x
\leq C\int_{\mathbb{R}^+}(w_{xx}^2+z_{xx}^2)\mathrm{d}x+C\int_{\mathbb{R}^+}(w_{x}^2+z_{x}^2)\mathrm{d}x+C\delta(1+t)^{-1}\int_{\mathbb{R}^+}z^2\widetilde\omega^2\mathrm{d}x.
\end{equation}
 Plugging \eqref{3.1.29} into \eqref{3.1.28}, it is easy to derive that
\begin{equation}\label{3.1.30}
\begin{split}
&\frac{\rm d}{{\rm d}t} \int_{\mathbb{R}^+}
   \left(\frac{\lambda{w_{t}}^2}{2} +\frac{K z_{t}^2}{2}-\mu w_{t}z_{t}\right) \mathrm{d}x
   +\frac{b\lambda}{8}\int_{\mathbb{R}^+}w_{xt}^2\mathrm{d}x+\frac{K a}{4}\int_{\mathbb{R}^+}z_{xt}^2\mathrm{d}x+\frac{1}{2}\int_{\mathbb{R}^+}{(\lambda w_{t}-\mu z_{t})^2}\mathrm{d}x\\
\leq&C\delta(1+t)^{-\frac{7}{2}}
    +C\int_{\mathbb{R}^+}(w_{xx}^2+z_{xx}^2)\mathrm{d}x+C\int_{\mathbb{R}^+}(w_{x}^2+z_{x}^2)\mathrm{d}x+C\delta(1+t)^{-1}\int_{\mathbb{R}^+}z^2\widetilde\omega^2\mathrm{d}x.
\end{split}
\end{equation}
Integrating the resulting inequality with respect to $t$, then taking $h = z$ in \eqref{3.1.11}, together with \eqref{3.1.15} and \eqref{3.1.19}, we can immediately obtain \eqref{3.1.25}. The proof of Lemma \ref{Lemma 3.6} is completed.
\end{proof}
\begin{lemma}\label{Lemma 3.7}
Under the assumptions of Theorem \ref{Thm 3.1}, we have
\begin{equation}\label{3.1.31}
\|w_{xx}(t)\|^{2}+\|z_{xx}(t)\|^{2}
\leq C\left(\left\|w_{0}\right\|_{2}^{2}+\left\|z_{0}\right\|_{2}^{2}+\delta\right),
\end{equation}
for $0 \leq t \leq T$.
\end{lemma}
\begin{proof}
From the equation $\eqref{3.1.2}_2$ and Lemma \ref{Lemma 3.1}, it is easy to obtain that
\begin{equation}\label{3.1.32}
\int_{\mathbb{R}^+}w_{xx}^2\mathrm{d}x\leq C \int_{\mathbb{R}^+}(w_t^2+w^2+z^2)\mathrm{d}x+C\delta.
\end{equation}
By using \eqref{3.1.15} and \eqref{3.1.25}, we can reach
\begin{equation}\label{3.1.33}
\|w_{xx}\|^2\leq C(\|w_{0}\|_{2}^{2}+\|z_{0}\|_{2}^{2}+\delta).
\end{equation}
Similarly, by using the equation $\eqref{3.1.2}_1$, we also have
\begin{equation}\label{3.1.34}
\int_{\mathbb{R}^+}z_{xx}^2\mathrm{d}x\leq C \int_{\mathbb{R}^+}(z_t^2+z^2+z_x^2+w_x^2+w_{xx}^2)\mathrm{d}x.
\end{equation}
Applying \eqref{3.1.15}, \eqref{3.1.19}, \eqref{3.1.25} and \eqref{3.1.33}, it holds that
\begin{equation}\label{3.1.35}
\|z_{xx}\|^2\leq C(\|w_{0}\|_{2}^{2}+\|z_{0}\|_{2}^{2}+\delta).
\end{equation}
Combining \eqref{3.1.33} and \eqref{3.1.35}, we obtain \eqref{3.1.31}. The proof of Lemma \ref{Lemma 3.7} is completed.
\end{proof}
From Lemmas \ref{Lemma 3.4}-\ref{Lemma 3.7}, one can get the desired inequality \eqref{3.1.4}. Therefore, applying the local existence result and the continuity argument, one can extend the local solution for problem \eqref{3.1.2}-\eqref{3.1.3} globally.

Finally, we have to show the desired large time behavior in Theorem \ref{Thm 3.1}.
Multiplying $\partial_x\eqref{3.1.2}_2$ by $w_x$, then using Lemma \ref{Lemma 3.1} and \eqref{3.1.14}, we have
\begin{equation}\label{3.1.36}
\begin{split}
\left|\frac{\mathrm{d}}{\mathrm{d} t}\left\|w_{x}(t)\right\|^{2}\right|
&\leq
C(\left\|w_{x}(t)\right\|^{2}+\left\|w_{x x}(t)\right\|^{2}+\left\|z_{x}(t)\right\|^{2})+C\delta^2(1+t)^{-\frac{5}{2}}+C\delta(1+t)^{-1}\|w_x(t)\|_{L^\infty}\\
&\leq C\delta(1+t)^{-2}+C(\left\|w_{x}(t)\right\|^{2}+\left\|w_{x x}(t)\right\|^{2}+\left\|z_{x}(t)\right\|^{2}).
\end{split}
\end{equation}
Hence, using the estimate \eqref{3.1.4} implies that
\begin{equation}\label{3.1.37}
\int_{0}^{\infty}\left(\left\|w_{x}(t)\right\|^{2}+\left|\frac{\mathrm{d}}{\mathrm{d} t}\left\|w_{x}(t)\right\|^{2}\right|\right) \mathrm{d} t<\infty.
\end{equation}
From Lemma \ref{Lemma 3.3}, we can derive
\begin{equation}\label{3.1.38}
\lim _{t \rightarrow+\infty}\left\|w_{x}(\cdot,t)\right\|=0.
\end{equation}
Similarly, it is easy to obtain that
\begin{equation}\label{3.1.39}
\int_{0}^{\infty}\left(\left\|z_{x}(t)\right\|^{2}+\left|\frac{\mathrm{d}}{\mathrm{d} t}\left\|z_{x}(t)\right\|^{2}\right|\right) \mathrm{d} t<\infty,
\end{equation}
then it follows that
\begin{equation}\label{3.1.40}
\lim _{t \rightarrow+\infty}\left\|z_{x}(\cdot,t)\right\|=0.
\end{equation}
Applying the Sobolev inequality, \eqref{3.1.38} and \eqref{3.1.40} easily leads to the large-time behavior \eqref{3.1.5} of the solution. This ends the proof of Theorem \ref{Thm 3.1}.
\subsection{The case of Neumann boundary condition}\label{S3.2}
\subsubsection{Reformulation of the problem and theorem}\label{S3.2.1}
We now turn to the problem \eqref{1.2} and \eqref{1.6} with the Neumann boundary condition \eqref{1.8}. As in the preceding subsection we first reformulate the problem \eqref{1.2} and \eqref{1.6}. Assume that the steady state of the  one-dimensional Keller-Segel model \eqref{1.2} is trivial, taking the form of
\begin{equation}\label{3.2.1}
u=u_+,\quad \rho=\rho_+,
\end{equation}
provided that
\begin{equation}\label{3.2.2}
\rho_+=\frac{\mu}{\lambda}u_+.
\end{equation}
Let $z=u-u_+$, $w=\rho-\rho_+$. Then $(w,z)$ satisfies
\begin{equation}\label{3.2.3}
\left\{\begin{array}{l}
z_t
   -az_{xx}
   +\kappa[(z+u_+)w_x]_x=0,\\[2mm]
w_t
    -bw_{xx}
    +\lambda w
    -\mu z=0,\\[2mm]
 \end{array}
        \right.
\end{equation}
with the initial data
\begin{equation}\label{3.2.4}
\left(w,z)\right|_{t=0}=\left(w_{0},z_{0}\right)(x) \rightarrow 0\quad
 \text {as}
  \quad
  x\rightarrow \infty,
\end{equation}
and the boundary condition
\begin{equation}\label{3.2.5}
\left(w_x,z_x)\right|_{x=0}=(0,0).
\end{equation}
\begin{theorem}\label{Thm 3.2} {\bf (Neumann boundary).}
Suppose that both $\delta_0:=\left|u_{+}\right|$ and $\left\|w_{0}\right\|_{2}+\left\|z_{0}\right\|_{2}$ are sufficiently small. Then there exists a unique time-global solution $(w,z)(x, t)$ of the initial-boundary value problem \eqref{3.2.3}-\eqref{3.2.5}, which satisfies
\begin{equation}\notag
w \in C^{i, \infty}\left([0, \infty) ; H^{2-i}\right), \quad i=0,1,2,~~~
z \in C^{i, \infty}\left([0, \infty) ; H^{2-i}\right), \quad i=0,1,2,	
\end{equation}
and
\begin{equation}\label{3.2.6}
\begin{split}
&\sum_{k=0}^{2}(1+t)^{k}\left(\|\partial_{x}^{k} w(t)\|^{2}+\|\partial_{x}^{k} z(t)\|^{2}\right)
  +(1+t)^{2}\left(\|w_t(t)\|^{2}+\|z_t(t)\|^{2}\right)
   \\
   &
    +
    \int_{0}^{t}\bigg[\sum_{j=0}^{2}(1+\tau)^{j}\left(\|\partial_{x}^{j+1} w(\tau)\|^{2}+\|\partial_{x}^{j+1} z(\tau)\|^{2}+\|\partial_x^j(\lambda w-\mu z)(\tau)\|^2\right)\\
    &+(1+\tau)^{2}\left(\|w_{xt}(\tau)\|^{2}+\|z_{xt}(\tau)\|^{2}+\|(\lambda w_t-\mu z_t)(\tau)\|^2\right)\bigg]{\rm d} \tau\\
     \leq&
       C\left(\|w_{0}\|_{2}^{2}+\|z_{0}\|_{2}^{2}\right).
\end{split}
\end{equation}
\end{theorem}
\subsubsection{Proof of Theorem \ref{Thm 3.2}}\label{S3.2.2}
In this subsection, we devote ourselves to the proof of Theorem \ref{Thm 3.2}. As long as {\it a priori} estimates is proved, Theorem \ref{Thm 3.2}} follows in the standard method by combining it with the local-in-time existence and uniqueness as well as the continuity argument. Therefore, in what follows we only estimate the solution $(w,z)(x,t)$, $0<t<T<\infty$, to the initial-boundary value problem \eqref{3.2.3}-\eqref{3.2.5} under the {\it a priori} assumption
\begin{equation}\label{3.2.7}
N(T):=\sup_{0\leq t\leq T}\left\{\sum_{k=0}^{2}(1+t)^{k}\left(\|\partial_{x}^{k} w(t)\|^{2}+\|\partial_{x}^{k} z(t)\|^{2}\right)\right\}\leq\varepsilon_0^2,
\end{equation}
for some $0<\varepsilon_0\ll 1$, and other details are omitted for simplicity.

By the Sobolev inequality, it is easy to deduce that
\begin{equation}\label{3.2.8}
\begin{split}
&\|\partial_{x}^{k} w(\cdot, t)\|_{L^{\infty}} \leq \sqrt2\varepsilon_0(1+t)^{-\frac{1}{4}-\frac{k}{2}}, \quad k=0,1,\\
&\|\partial_{x}^{k} z(\cdot, t)\|_{L^{\infty}} \leq \sqrt2\varepsilon_0(1+t)^{-\frac{1}{4}-\frac{k}{2}}, \quad k=0,1,
 \end{split}
\end{equation}
which will be frequently used in the sequel.

It can be checked that
\begin{equation}\label{3.2.9}
z_x(0,t)=z_{xt}(0,t)=w_x(0,t)=w_{xt}(0,t)=w_{xxx}(0,t)=0,
\end{equation}
where we have used the equation $\eqref{3.2.3}_2$.

In fact, Theorem \ref{Thm 3.2} will be proved by the following series of lemmas, and the estimates obtained below are formally quite similar to those in Section \ref{S2.3}.
\begin{lemma}\label{Lemma 3.8}
Under the assumptions of Theorem \ref{Thm 3.2}, we have
\begin{equation}\label{3.2.10}
\|w(t)\|^{2}+\|z(t)\|^{2}+\int_{0}^{t}\left(\left\|w_{x}(\tau)\right\|^{2}+\|z_{x}(\tau)\|^{2}+\|(\lambda w-\mu z)(\tau)\|^2\right) \mathrm{d} \tau \leq C\left(\left\|w_{0}\right\|^{2}+\left\|z_{0}\right\|^{2}\right),
\end{equation}
for $0 \leq t \leq T$.
\end{lemma}
\begin{proof}
Similarly, in order to produce good term by using the structure of reformulated equations, we can get from $\int_{\mathbb{R}^+}\lambda w\times\eqref{3.2.3}_2\mathrm{d}x-\int_{\mathbb{R}^+}\mu z\times\eqref{3.2.3}_2\mathrm{d}x$ that
\begin{equation}\label{3.2.11}
\begin{split}
&\frac{\rm d}{{\rm d}t} \int_{\mathbb{R}^+}
   \frac{\lambda{w}^2}{2}\mathrm{d}x
   +b\lambda\int_{\mathbb{R}^+} w_x^2\mathrm{d}x
   +\int_{\mathbb{R}^+} (\lambda w-\mu z)^2\mathrm{d}x\\
&=\mu\int_{\mathbb{R}^+}w_tz\mathrm{d}x+b\mu\int_{\mathbb{R}^+} {w_x}{z_x}\mathrm{d}x-bw_x(0,t)(\lambda w-\mu z)(0,t).
\end{split}
\end{equation}
By applying integration by parts and using the equation $\eqref{3.2.3}_1$, one can obtain
\begin{equation}\label{3.2.12}
\begin{split}
\mu\int_{\mathbb{R}^+}w_tz\mathrm{d}x
&=\frac{\rm d}{{\rm d}t}\int_{\mathbb{R}^+}\mu wz\mathrm{d}x-\mu\int_{\mathbb{R}^+}wz_t\mathrm{d}x\\
&=\frac{\rm d}{{\rm d}t}\int_{\mathbb{R}^+}\mu wz\mathrm{d}x-\mu\int_{\mathbb{R}^+}w\bigg\{az_{xx} -\kappa[(z+u_+)w_x]_x\bigg\}\mathrm{d}x\\
&=\frac{\rm d}{{\rm d}t}\int_{\mathbb{R}^+}\mu wz\mathrm{d}x+a\mu\int_{\mathbb{R}^+}w_xz_x\mathrm{d}x-\kappa\mu\int_{\mathbb{R}^+}(z+u_+)w_{x}^2\mathrm{d}x\\
&~~~~+\mu w(0,t)\{az_x-\kappa[(z+u_+)w_x]\}(0,t)\\
&\leq\frac{\rm d}{{\rm d}t}\int_{\mathbb{R}^+}\mu wz\mathrm{d}x+a\mu\int_{\mathbb{R}^+}w_xz_x\mathrm{d}x+\kappa\mu\int_{\mathbb{R}^+}(\|z\|_{L^{\infty}}+|u_+|)w_{x}^2\mathrm{d}x\\
&\leq\frac{\rm d}{{\rm d}t}\int_{\mathbb{R}^+}\mu wz\mathrm{d}x+a\mu\int_{\mathbb{R}^+}w_xz_x\mathrm{d}x+\kappa\mu (\varepsilon_0+\delta_0)\int_{\mathbb{R}^+}w_{x}^2\mathrm{d}x,
\end{split}
\end{equation}
where we have used \eqref{3.2.8} and \eqref{3.2.9}.

Putting \eqref{3.2.12} into \eqref{3.2.11}, then applying \eqref{3.2.9} and Young inequality, we can get
\begin{equation}\label{3.2.13}
\begin{split}
&\frac{\rm d}{{\rm d}t}\int_{\mathbb{R}^+} \left(\frac{\lambda w^2}{2}-\mu wz\right) \mathrm{d}x
+\frac{b\lambda}{4}\int_{\mathbb{R}^+}w_x^2\mathrm{d}x
+\int_{\mathbb{R}^+}(\lambda w-\mu z)^2\mathrm{d}x
\leq C\int_{\mathbb{R}^+}z_{x}^2\mathrm{d}x.
\end{split}
\end{equation}
Next, multiplying $\eqref{3.2.3}_1$ by $K z$ ($K$ is sufficiently large) and integrating it with respect to $x$ over $\mathbb{R}^+$, we have from \eqref{3.2.8}-\eqref{3.2.9} that
\begin{equation}\label{3.2.14}
\begin{split}
\frac{\rm d}{{\rm d}t} \int_{\mathbb{R}^+}\frac{ K z^2}{2}\mathrm{d}x +\frac{K a}{2}\int_{\mathbb{R}^+}z_{x}^2\mathrm{d}x
\leq& C\int_{\mathbb{R}^+}(\|z\|_{L^{\infty}}+|u_+|)w_{x}^2\mathrm{d}x\\
\leq& C(\varepsilon_0+\delta_0)\int_{\mathbb{R}^+}w_{x}^2\mathrm{d}x.
\end{split}
\end{equation}
Hence, combining \eqref{3.2.13} and \eqref{3.2.14}, we obtain
\begin{equation}\label{3.2.15}
\frac{\rm d}{{\rm d}t} \int_{\mathbb{R}^+}
   \left(\frac{\lambda w^2}{2} +\frac{K z^2}{2}-\mu wz\right) \mathrm{d}x
   +\frac{b\lambda}{8}\int_{\mathbb{R}^+}w_x^2\mathrm{d}x+\frac{K a}{4}\int_{\mathbb{R}^+}z_x^2\mathrm{d}x
   +\int_{\mathbb{R}^+}{(\lambda w-\mu z)^2}\mathrm{d}x
   \leq 0.
\end{equation}
Integrating the above inequality with respect to $t$, we conclude \eqref{3.2.10}. The proof of Lemma \ref{Lemma 3.8} is completed.
\end{proof}
\begin{lemma}\label{Lemma 3.9}
Under the assumptions of Theorem \ref{Thm 3.2}, we have
\begin{equation}\label{3.2.16}
\begin{split}
(1+t)&(\|w_{x}(t)\|^{2}+\|z_{x}(t)\|^{2})+\int_{0}^{t}(1+\tau)\left(\left\|w_{x x}(\tau)\right\|^{2}+\|z_{x x}(\tau)\|^{2}+\|(\lambda w_x-\mu z_x)(\tau)\|^2\right) \mathrm{d} \tau\\
&\leq C\left(\left\|w_{0}\right\|_{1}^{2}+\left\|z_{0}\right\|_{1}^{2}\right),
\end{split}
\end{equation}
for $0 \leq t \leq T$.
\end{lemma}
\begin{proof}
Similarly, by $\int_{\mathbb{R}^+}\lambda w_x\times\partial_x\eqref{3.2.3}_2\mathrm{d}x-\int_{\mathbb{R}^+}\mu z_x\times\partial_x\eqref{3.2.3}_2\mathrm{d}x$, we have
\begin{equation}\label{3.2.17}
\begin{split}
&\frac{\rm d}{{\rm d}t} \int_{\mathbb{R}^+}
   \frac{\lambda{w_x}^2}{2}\mathrm{d}x
   +b\lambda\int_{\mathbb{R}^+} w_{xx}^2\mathrm{d}x
   +\int_{\mathbb{R}^+} (\lambda w_x-\mu z_x)^2\mathrm{d}x\\
&=\mu\int_{\mathbb{R}^+}w_{xt}z_x\mathrm{d}x+b\mu\int_{\mathbb{R}^+} {w_{xx}}{z_{xx}}\mathrm{d}x-bw_{xx}(0,t)(\lambda w_x-\mu z_x)(0,t).
\end{split}
\end{equation}
Now we estimate the first term in the right hand of \eqref{3.2.17}. By applying integration by parts and using $\partial_x\eqref{3.2.3}_1$, together with \eqref{3.2.9}, one can obtain
\begin{equation}\label{3.2.18}
\begin{split}
\mu\int_{\mathbb{R}^+}w_{xt}z_x\mathrm{d}x
&=\frac{\rm d}{{\rm d}t}\int_{\mathbb{R}^+}\mu w_xz_x\mathrm{d}x-\mu\int_{\mathbb{R}^+}w_xz_{xt}\mathrm{d}x\\
&=\frac{\rm d}{{\rm d}t}\int_{\mathbb{R}^+}\mu w_xz_x\mathrm{d}x-\mu\int_{\mathbb{R}^+}w_x\bigg\{az_{xxx} -\kappa[(z+u_+)w_x]_{xx}\bigg\}\mathrm{d}x\\
&=\frac{\rm d}{{\rm d}t}\int_{\mathbb{R}^+}\mu w_xz_x\mathrm{d}x+a\mu\int_{\mathbb{R}^+}w_{xx}z_{xx}\mathrm{d}x-\kappa\mu\int_{\mathbb{R}^+}w_{xx}[(z+u_+)w_x]_x\mathrm{d}x\\
&~~~~+\mu w_{x}(0,t)z_{t}(0,t)\\
&=\frac{\rm d}{{\rm d}t}\int_{\mathbb{R}^+}\mu w_xz_x\mathrm{d}x+a\mu\int_{\mathbb{R}^+}w_{xx}z_{xx}\mathrm{d}x-\kappa\mu\int_{\mathbb{R}^+}w_{xx}z_xw_x\mathrm{d}x\\
&~~~~-\kappa\mu\int_{\mathbb{R}^+}(z+u_+)w_{xx}^2\mathrm{d}x.
\end{split}
\end{equation}
Firstly, by using Young inequality and \eqref{3.2.8}, we have
\begin{equation}\label{3.2.19}
\begin{split}
-\kappa\mu\int_{\mathbb{R}^+}w_{xx}z_xw_x\mathrm{d}x
&\leq \eta\int_{\mathbb{R}^+}w_{xx}^2\mathrm{d}x+C_{\eta}\int_{\mathbb{R}^+}\|z_x\|_{L^{\infty}}^2w_x^2\mathrm{d}x\\
&\leq \eta\int_{\mathbb{R}^+}w_{xx}^2\mathrm{d}x+C_{\eta}\varepsilon_0^2(1+t)^{-\frac{3}{2}}\int_{\mathbb{R}^+}w_x^2\mathrm{d}x.
\end{split}
\end{equation}
Secondly, it is easy to derive that
\begin{equation}\label{3.2.20}
\begin{split}
-\kappa\mu\int_{\mathbb{R}^+}(z+u_+)w_{xx}^2\mathrm{d}x
\leq C\int_{\mathbb{R}^+}(\|z\|_{L^{\infty}}+|u_+|)w_{xx}^2\mathrm{d}x
\leq C(\varepsilon_0+\delta_0)\int_{\mathbb{R}^+}w_{xx}^2\mathrm{d}x.
\end{split}
\end{equation}
Substituting \eqref{3.2.19} and \eqref{3.2.20} into \eqref{3.2.18}, and summing the resulting inequality to \eqref{3.2.17}, then by using Young inequality and \eqref{3.2.9}, we have
\begin{equation}\label{3.2.21}
\begin{split}
&\frac{\rm d}{{\rm d}t}\int_{\mathbb{R}^+} \left(\frac{\lambda w_x^2}{2}-\mu w_xz_x\right) \mathrm{d}x
+\frac{b\lambda}{4}\int_{\mathbb{R}^+}w_{xx}^2\mathrm{d}x
+\int_{\mathbb{R}^+}{(\lambda w_x-\mu z_x)^2}\mathrm{d}x\\
\leq& C(1+t)^{-\frac{3}{2}}\int_{\mathbb{R}^+}w_{x}^2\mathrm{d}x
   +C\int_{\mathbb{R}^+}z_{xx}^2\mathrm{d}x.
\end{split}
\end{equation}
Multiplying $\partial_x\eqref{3.2.3}_1$ by $K z_x$ ($K$ is sufficiently large) and integrating it with respect to $x$ over $\mathbb{R}^+$, then similar to the treatment of \eqref{3.2.19} and \eqref{3.2.20}, we can get from \eqref{3.2.9} that
\begin{equation}\label{3.2.22}
\frac{\rm d}{{\rm d}t} \int_{\mathbb{R}^+}\frac{ K z_x^2}{2}\mathrm{d}x +\frac{K a}{2}\int_{\mathbb{R}^+}z_{xx}^2\mathrm{d}x
\leq C(\varepsilon_0+\delta_0)\int_{\mathbb{R}^+}w_{xx}^2\mathrm{d}x+C(1+t)^{-\frac{3}{2}}\int_{\mathbb{R}^+}w_x^2\mathrm{d}x.
\end{equation}
Therefore, combining \eqref{3.2.21} and \eqref{3.2.22}, we obtain
\begin{equation}\label{3.2.23}
\begin{split}
&\frac{\rm d}{{\rm d}t} \int_{\mathbb{R}^+}
   \left(\frac{\lambda w_x^2}{2} +\frac{K z_x^2}{2}-\mu w_xz_x\right) \mathrm{d}x
   +\frac{b\lambda}{8}\int_{\mathbb{R}^+}w_{xx}^2\mathrm{d}x+\frac{K a}{4}\int_{\mathbb{R}^+}z_{xx}^2\mathrm{d}x
   +\int_{\mathbb{R}^+}{(\lambda w_x-\mu z_x)^2}\mathrm{d}x\\
   \leq
&C(1+t)^{-\frac{3}{2}}\int_{\mathbb{R}^+}w_x^2\mathrm{d}x.
\end{split}
\end{equation}
Integrating \eqref{3.2.23} over $(0,t)$, we get
\begin{equation}\label{3.2.24}
\begin{split}
&\|w_{x}(t)\|^{2}+\|z_{x}(t)\|^{2}+\int_{0}^{t}\left(\left\|w_{x x}(\tau)\right\|^{2}+\|z_{x x}(\tau)\|^{2}+\|(\lambda w_x-\mu z_x)(\tau)\|^2\right) \mathrm{d} \tau\\
\leq& C\left(\left\|w_{0}\right\|_{1}^{2}+\left\|z_{0}\right\|_1^{2}\right).
\end{split}
\end{equation}
Multiplying \eqref{3.2.23} by $(1 + t)$, and integrating it with respect to $t$, by applying \eqref{3.2.10}, one can immediately obtain \eqref{3.2.16}. The proof of Lemma \ref{Lemma 3.9} is completed.
\end{proof}
\begin{lemma}\label{Lemma 3.10}
Under the assumptions of Theorem \ref{Thm 3.2}, we have
\begin{equation}\label{3.2.25}
\begin{split}
&(1+t)^{2}(\|w_{x x}(t)\|^{2}+\|z_{x x}(t)\|^{2})\\
&+\int_{0}^{t}(1+\tau)^{2}\left(\left\|w_{x x x}(\tau)\right\|^{2}+\|z_{x x x}(\tau)\|^{2}+\|(\lambda w_{xx}-\mu z_{xx})(\tau)\|^2\right) \mathrm{d} \tau\\
\leq& C\left(\left\|w_{0}\right\|_{2}^{2}+\left\|z_{0}\right\|_{2}^{2}\right),
\end{split}
\end{equation}
for $0 \leq t \leq T$.
\end{lemma}
\begin{proof}
Similarly, from $\int_{\mathbb{R}^+}\lambda w_{xx}\times\partial_x^2\eqref{3.2.3}_2\mathrm{d}x-\int_{\mathbb{R}^+}\mu z_{xx}\times\partial_x^2\eqref{3.2.3}_2\mathrm{d}x$, then by applying $\partial_x^2\eqref{3.2.3}_1$ and \eqref{3.2.9}, we have
\begin{equation}\label{3.2.26}
\begin{split}
&\frac{\rm d}{{\rm d}t} \int_{\mathbb{R}^+}
   \left(\frac{\lambda w_{xx}^2}{2}-\mu w_{xx}z_{xx}\right)\mathrm{d}x
   +{b\lambda}\int_{\mathbb{R}^+}w_{xxx}^2\mathrm{d}x
   +\int_{\mathbb{R}^+} (\lambda w_{xx}-\mu z_{xx})^2\mathrm{d}x\\
&=(a+b)\mu\int_{\mathbb{R}^+}w_{xxx}z_{xxx}\mathrm{d}x-\kappa\mu\int_{\mathbb{R}^+}w_{xxx}[(z+u_+)w_x]_{xx}\mathrm{d}x-bw_{xxx}(0,t)(\lambda w_{xx}-\mu z_{xx})(0,t)\\
&~~~~+\mu w_{xx}(0,t)z_{xt}(0,t)\\
&=(a+b)\mu\int_{\mathbb{R}^+}w_{xxx}z_{xxx}\mathrm{d}x
-\kappa\mu\int_{\mathbb{R}^+}w_{x x x}w_{x}z_{x x}\mathrm{d}x
-2\kappa\mu\int_{\mathbb{R}^+}w_{x x x}z_{x}w_{x x}\mathrm{d}x\\
&~~~~-\kappa\mu\int_{\mathbb{R}^+}(z+u_+)w_{x x x}^2\mathrm{d}x.
\end{split}
\end{equation}
Firstly, applying Young inequality and \eqref{3.2.8}, one gets
\begin{equation}\label{3.2.27}
\begin{split}
-\kappa\mu\int_{\mathbb{R}^+}w_{x x x}w_{x}z_{x x}\mathrm{d}x
&\leq \eta\int_{\mathbb{R}^+}w_{xxx}^2\mathrm{d}x+C_{\eta}\int_{\mathbb{R}^+}\|w_x\|_{L^{\infty}}^2z_{xx}^2\mathrm{d}x\\
&\leq \eta\int_{\mathbb{R}^+}w_{xxx}^2\mathrm{d}x+C_{\eta}\varepsilon_0^2(1+t)^{-\frac{3}{2}}\int_{\mathbb{R}^+}z_{xx}^2\mathrm{d}x,
\end{split}
\end{equation}
and
\begin{equation}\label{3.2.28}
\begin{split}
-2\kappa\mu\int_{\mathbb{R}^+}w_{x x x}z_{x}w_{x x}\mathrm{d}x
&\leq \eta\int_{\mathbb{R}^+}w_{xxx}^2\mathrm{d}x+C_{\eta}\int_{\mathbb{R}^+}\|z_x\|_{L^{\infty}}^2w_{xx}^2\mathrm{d}x\\
&\leq \eta\int_{\mathbb{R}^+}w_{xxx}^2\mathrm{d}x+C_{\eta}\varepsilon_0^2(1+t)^{-\frac{3}{2}}\int_{\mathbb{R}^+}w_{xx}^2\mathrm{d}x.
\end{split}
\end{equation}
Secondly, similar to treatment of \eqref{3.2.20}, we have
\begin{equation}\label{3.2.29}
\begin{split}
-\kappa\mu\int_{\mathbb{R}^+}(z+u_+)w_{x x x}^2\mathrm{d}x
\leq C\int_{\mathbb{R}^+}(\|z\|_{L^{\infty}}+|u_+|)w_{xxx}^2\mathrm{d}x
\leq C(\varepsilon_0+\delta_0)\int_{\mathbb{R}^+}w_{xxx}^2\mathrm{d}x.
\end{split}
\end{equation}
Putting \eqref{3.2.27}-\eqref{3.2.29} into \eqref{3.2.26}, we have
\begin{equation}\label{3.2.30}
\begin{split}
&\frac{\rm d}{{\rm d}t}\int_{\mathbb{R}^+} \left(\frac{\lambda w_{xx}^2}{2}-\mu w_{xx}z_{xx}\right) \mathrm{d}x
+\frac{b\lambda}{4}\int_{\mathbb{R}^+}w_{xxx}^2\mathrm{d}x
+\int_{\mathbb{R}^+}{(\lambda w_{xx}-\mu z_{xx})^2}\mathrm{d}x\\
\leq&
   C(1+t)^{-\frac{3}{2}}\int_{\mathbb{R}^+}(w_{xx}^2+z_{xx}^2)\mathrm{d}x
   +C\int_{\mathbb{R}^+}z_{xxx}^2\mathrm{d}x.
\end{split}
\end{equation}
Next, multiplying $\partial_x^2\eqref{3.2.3}_1$ by $K z_{xx}$ ($K$ is sufficiently large) and integrating it with respect to $x$ over $\mathbb{R}^+$, then similar to the treatment of \eqref{3.2.27}, \eqref{3.2.28} and \eqref{3.2.29}, we get
\begin{equation}\label{3.2.31}
\begin{split}
\frac{\rm d}{{\rm d}t} \int_{\mathbb{R}^+}\frac{ K z_{xx}^2}{2}\mathrm{d}x+\frac{K a}{2}\int_{\mathbb{R}^+}z_{xxx}^2\mathrm{d}x
&\leq C(\varepsilon_0+\delta_0)\int_{\mathbb{R}^+}w_{xxx}^2\mathrm{d}x+C(1+t)^{-\frac{3}{2}}\int_{\mathbb{R}^+}(w_{xx}^2+z_{xx}^2)\mathrm{d}x\\
&~~~~-K z_{xx}(0,t)z_{xt}(0,t)\\
&\leq C(\varepsilon_0+\delta_0)\int_{\mathbb{R}^+}w_{xxx}^2\mathrm{d}x+C(1+t)^{-\frac{3}{2}}\int_{\mathbb{R}^+}(w_{xx}^2+z_{xx}^2)\mathrm{d}x,
\end{split}
\end{equation}
where in the last inequality we have used \eqref{3.2.9}.

Combining \eqref{3.2.30} and \eqref{3.2.31}, it follows that
\begin{equation}\label{3.2.32}
\begin{split}
&\frac{\rm d}{{\rm d}t} \int_{\mathbb{R}^+}
   \left(\frac{\lambda w_{xx}^2}{2} +\frac{K z_{xx}^2}{2}-\mu w_{xx}z_{xx}\right) \mathrm{d}x
   +\frac{b\lambda}{8}\int_{\mathbb{R}^+}w_{xxx}^2\mathrm{d}x+\frac{K a}{4}\int_{\mathbb{R}^+}z_{xxx}^2\mathrm{d}x\\
   &+\int_{\mathbb{R}^+}{(\lambda w_{xx}-\mu z_{xx})^2}\mathrm{d}x\\
   \leq
&C(1+t)^{-\frac{3}{2}}\int_{\mathbb{R}^+}(w_{xx}^2+z_{xx}^2)\mathrm{d}x.
\end{split}
\end{equation}
Integrating \eqref{3.2.32} over $(0,t)$ and using \eqref{3.2.16}, we have
\begin{equation}\label{3.2.33}
\begin{split}
&\|w_{x x}(t)\|^{2}+\|z_{x x}(t)\|^{2}+\int_{0}^{t}\left(\left\|w_{x x x}(\tau)\right\|^{2}+\|z_{x x x}(\tau)\|^{2}+\|(\lambda w_{xx}-\mu z_{xx})(\tau)\|^2\right) \mathrm{d} \tau \\
\leq& C\left(\left\|w_{0}\right\|_{2}^{2}+\left\|z_{0}\right\|_{2}^{2}\right).
\end{split}
\end{equation}
Multiplying \eqref{3.2.32} by $(1+t)^2$, integrating with respect to $t$, then using \eqref{3.2.16}, we obtain \eqref{3.2.25}. The proof of Lemma \ref{Lemma 3.10} is completed.
\end{proof}
\begin{lemma}\label{Lemma 3.11}
Under the assumptions of Theorem \ref{Thm 3.2}, we have
\begin{equation}\label{3.2.34}
\begin{split}
&(1+t)^{2}(\|w_{t}(t)\|^{2}+\|z_{t}(t)\|^{2})\\
&+\int_{0}^{t}(1+\tau)^{2}\left(\left\|w_{xt}(\tau)\right\|^{2}+\|z_{xt}(\tau)\|^{2}+\|(\lambda w_{t}-\mu z_{t})(\tau)\|^2\right) \mathrm{d} \tau\\
\leq& C\left(\left\|w_{0}\right\|_{2}^{2}+\left\|z_{0}\right\|_{2}^{2}\right),
\end{split}
\end{equation}
for $0 \leq t \leq T$.
\end{lemma}
\begin{proof}
Similarly, from $\int_{\mathbb{R}^+}\lambda w_{t}\times\partial_t\eqref{3.2.3}_2\mathrm{d}x-\int_{\mathbb{R}^+}\mu z_{t}\times\partial_t\eqref{3.2.3}_2\mathrm{d}x$, from $\partial_t\eqref{3.2.3}_1$ and  \eqref{3.2.9}, we can obtain
\begin{equation}\label{3.2.35}
\begin{split}
&\frac{\rm d}{{\rm d}t} \int_{\mathbb{R}^+}
   \left(\frac{\lambda w_{t}^2}{2}-\mu w_{t}z_{t}\right)\mathrm{d}x
   +{b\lambda}\int_{\mathbb{R}^+}w_{xt}^2\mathrm{d}x
   +\int_{\mathbb{R}^+} (\lambda w_{t}-\mu z_{t})^2\mathrm{d}x\\
&=(a+b)\mu\int_{\mathbb{R}^+}w_{xt}z_{xt}\mathrm{d}x-\kappa\mu\int_{\mathbb{R}^+}w_{xt}[(z+u_+)w_x]_{t}\mathrm{d}x-bw_{xt}(0,t)(\lambda w_{t}-\mu z_{t})(0,t)\\
&~~~~+\mu w_{t}(0,t)\{az_{xt}-\kappa[(z+u_+)w_x]_t\}(0,t)\\
&=(a+b)\mu\int_{\mathbb{R}^+}w_{xxx}z_{xxx}\mathrm{d}x
-\kappa\mu\int_{\mathbb{R}^+}w_{xt}w_{x}z_{t}\mathrm{d}x
-\kappa\mu\int_{\mathbb{R}^+}(z+u_+)w_{xt}^2\mathrm{d}x.
\end{split}
\end{equation}
Firstly, using Young inequality and \eqref{3.2.8}, it follows that
\begin{equation}\label{3.2.36}
\begin{split}
-\kappa\mu\int_{\mathbb{R}^+}w_{xt}w_{x}z_{t}\mathrm{d}x
&\leq \eta\int_{\mathbb{R}^+}w_{xt}^2\mathrm{d}x+C_{\eta}\int_{\mathbb{R}^+}\|w_x\|_{L^{\infty}}^2z_{t}^2\mathrm{d}x\\
&\leq \eta\int_{\mathbb{R}^+}w_{xt}^2\mathrm{d}x+C_{\eta}\varepsilon_0^2(1+t)^{-\frac{3}{2}}\int_{\mathbb{R}^+}z_{t}^2\mathrm{d}x.
\end{split}
\end{equation}
Next, by using the similar method as \eqref{3.2.29}, we have
\begin{equation}\label{3.2.37}
\begin{split}
-\kappa\mu\int_{\mathbb{R}^+}(z+u_+)w_{xt}^2\mathrm{d}x
\leq C\int_{\mathbb{R}^+}(\|z\|_{L^{\infty}}+|u_+|)w_{xt}^2\mathrm{d}x
\leq C(\varepsilon_0+\delta_0)\int_{\mathbb{R}^+}w_{xt}^2\mathrm{d}x.
\end{split}
\end{equation}
Taking \eqref{3.2.36} and \eqref{3.2.37} into \eqref{3.2.35}, and using Young inequality, we can conclude that
\begin{equation}\label{3.2.38}
\begin{split}
&\frac{\rm d}{{\rm d}t}\int_{\mathbb{R}^+} \left(\frac{\lambda w_{t}^2}{2}-\mu w_{t}z_{t}\right) \mathrm{d}x
+\frac{b\lambda}{4}\int_{\mathbb{R}^+}w_{xt}^2\mathrm{d}x
+\int_{\mathbb{R}^+}{(\lambda w_{t}-\mu z_{t})^2}\mathrm{d}x\\
\leq&
   C\varepsilon_0(1+t)^{-\frac{3}{2}}\int_{\mathbb{R}^+}z_{t}^2\mathrm{d}x
   +C\int_{\mathbb{R}^+}z_{xt}^2\mathrm{d}x.
\end{split}
\end{equation}
Now, we try to treat the second term in the right-hand side of the above inequality. Multiplying $\partial_t\eqref{3.2.3}_1$ by $K z_{t}$ ($K$ is sufficiently large) and integrating it with respect to $x$ over $\mathbb{R}^+$, we can derive
\begin{equation}\label{3.2.39}
\begin{split}
\frac{\rm d}{{\rm d}t} \int_{\mathbb{R}^+}\frac{ K z_{t}^2}{2}\mathrm{d}x+\frac{K a}{2}\int_{\mathbb{R}^+}z_{xt}^2\mathrm{d}x
&\leq C(\varepsilon_0+\delta_0)\int_{\mathbb{R}^+}w_{xt}^2\mathrm{d}x+C\varepsilon_0(1+t)^{-\frac{3}{2}}\int_{\mathbb{R}^+}z_{t}^2\mathrm{d}x\\
&~~~~+K z_{t}(0,t)\{-az_{xt}+\kappa[(z+u_+)w_x]_t\}(0,t)\\
&\leq C(\varepsilon_0+\delta_0)\int_{\mathbb{R}^+}w_{xt}^2\mathrm{d}x+C\varepsilon_0(1+t)^{-\frac{3}{2}}\int_{\mathbb{R}^+}z_{t}^2\mathrm{d}x,
\end{split}
\end{equation}
Combining \eqref{3.2.38} and \eqref{3.2.39}, one can immediately obtain
\begin{equation}\label{3.2.40}
\begin{split}
&\frac{\rm d}{{\rm d}t} \int_{\mathbb{R}^+}
   \left(\frac{\lambda w_{t}^2}{2} +\frac{K z_{t}^2}{2}-\mu w_{t}z_{t}\right) \mathrm{d}x
   +\frac{b\lambda}{8}\int_{\mathbb{R}^+}w_{xt}^2\mathrm{d}x+\frac{K a}{4}\int_{\mathbb{R}^+}z_{xt}^2\mathrm{d}x
   +\int_{\mathbb{R}^+}{(\lambda w_{t}-\mu z_{t})^2}\mathrm{d}x\\
   \leq&
C\varepsilon_0(1+t)^{-\frac{3}{2}}\int_{\mathbb{R}^+}z_{t}^2\mathrm{d}x.
\end{split}
\end{equation}
By using the equation $\eqref{3.2.3}_1$ and \eqref{3.2.8}, it follows that
\begin{equation}\label{3.2.41}
\int_{\mathbb{R}^+}z_{t}^2\mathrm{d}x
\leq C\int_{\mathbb{R}^+}z_{xx}^2\mathrm{d}x+C\int_{\mathbb{R}^+}w_{xx}^2\mathrm{d}x+C(1+t)^{-\frac{3}{2}}\int_{\mathbb{R}^+}w_{x}^2\mathrm{d}x.
\end{equation}
Noting that $(\lambda w_{t})^2\leq 2(\lambda w_t-\mu z_t)^2+2(\mu z_t)^2$, combining \eqref{3.2.40} and \eqref{3.2.41} and choosing $\varepsilon_0$ small enough yields
\begin{equation}\label{3.2.42}
\begin{split}
&\frac{\rm d}{{\rm d}t} \int_{\mathbb{R}^+}
   \left(\frac{\lambda{w_{t}}^2}{2} +\frac{K z_{t}^2}{2}-\mu w_{t}z_{t}\right) \mathrm{d}x
   +\frac{b\lambda}{8}\int_{\mathbb{R}^+}w_{xt}^2\mathrm{d}x+\frac{K a}{4}\int_{\mathbb{R}^+}z_{xt}^2\mathrm{d}x+\frac{\lambda^2}{2}\int_{\mathbb{R}^+}w_t^2\mathrm{d}x\\
   &+\frac{1}{2}\int_{\mathbb{R}^+}z_t^2\mathrm{d}x
\leq C\int_{\mathbb{R}^+}(w_{xx}^2+z_{xx}^2)\mathrm{d}x+C(1+t)^{-\frac{3}{2}}\int_{\mathbb{R}^+}w_{x}^2\mathrm{d}x.
\end{split}
\end{equation}
Integrating \eqref{3.2.42} with respect to $t$, then employing \eqref{3.2.10} and \eqref{3.2.16}, we obtain
\begin{equation}\label{3.2.43}
\begin{split}
&\|w_{t}(t)\|^{2}+\|z_{t}(t)\|^{2}+\int_{0}^{t}\left(\left\|w_{xt}(\tau)\right\|^{2}+\|z_{xt}(\tau)\|^{2}+\|w_{t}(\tau)\|^2+\|z_{t}(\tau)\|^2\right) \mathrm{d} \tau\\
\leq& C\left(\left\|w_{0}\right\|_{2}^{2}+\left\|z_{0}\right\|_{2}^{2}\right).
\end{split}
\end{equation}
Multiplying \eqref{3.2.42} by $(1 + t)$, we integrate it to obtain
\begin{equation}\label{3.2.44}
\begin{split}
&(1+t)(\|w_{t}(t)\|^{2}+\|z_{t}(t)\|^{2})+\int_{0}^{t}(1+\tau)\left(\left\|w_{xt}(\tau)\right\|^{2}+\|z_{xt}(\tau)\|^{2}+\|w_{t}(\tau)\|^2+\|z_{t}(\tau)\|^2\right) \mathrm{d} \tau\\
\leq& C\left(\left\|w_{0}\right\|_{2}^{2}+\left\|z_{0}\right\|_{2}^{2}\right).
\end{split}
\end{equation}
At last, multiplying \eqref{3.2.40} by $(1+t)^2$, then using \eqref{3.2.44}, we derive \eqref{3.2.34}. The proof of Lemma \ref{Lemma 3.11} is completed.
\end{proof}
Recalling Lemmas \ref{Lemma 3.8}-\ref{Lemma 3.11}, we complete the proof of Theorem \ref{Thm 3.2}.
}

\vspace{6mm}

\noindent {\bf Acknowledgements:} The research was supported by the National Natural Science Foundation of China $\#$12171160, 11771150, 11831003 and Guangdong Basic and Applied Basic Research Foundation $\#$2020B1515310015.

{\small

\bibliographystyle{plain}

}

\end{document}